\documentclass[11pt]{amsart}

\usepackage{amsfonts,amssymb,amsmath, amsthm}
\usepackage{graphicx}
\usepackage{systeme}
\usepackage{pgf,tikz,pgfplots}
\usepackage{kotex}
\usepackage[margin=1in]{geometry}

\pgfplotsset{compat=1.15}
\usepgfplotslibrary{fillbetween}
\usepackage{mathrsfs}
\usepackage{comment}
\usetikzlibrary{arrows}
\usetikzlibrary{calc}
\newtheorem{theorem}{Theorem}[section]
\newtheorem{lemma}{Lemma}[section]

\newtheorem{proposition}{Proposition}[section]
\newtheorem{remark}{Remark}[section]

\theoremstyle{definition}

\theoremstyle{remark}

\numberwithin{equation}{section}

\newcommand{\R}{\mathbb{R}}

\newcommand{\sX}{\dot{\bold{X}}}
\newcommand{\bB}{\bold{B}}
\newcommand{\bG}{\bold{G}}
\newcommand{\bS}{\bold{S}}
\newcommand{\bD}{\bold{D}}
\newcommand{\bX}{\bold{X}}

\newcommand{\tv}{\widetilde{v}}

\newcommand{\tu}{\widetilde{u}}

\newcommand{\e}{\varepsilon}
\newcommand{\pa}{\partial}

\newcommand{\bth}{\bar{\theta}}
\def\di{\displaystyle}

\usepackage{float}
\restylefloat{table}
\restylefloat{figure}
\usepackage{dsfont}

\title[Navier-Stokes-Fourier equations]{Time-asymptotic stability of composite wave of viscous shocks and viscous contact wave for Navier-Stokes-Fourier equations}

\author[Huang]{Xushan Huang}
\address[Xushan Huang]{\newline Department of Mathematical Sciences \newline Korea Advanced Institute of Science and Technology, Daejeon  34141, Republic of Korea}
\email{xushanhuang@kaist.ac.kr}

\author[Lee]{Hobin Lee}
\address[Hobin Lee]{\newline Department of Mathematical Sciences \newline Korea Advanced Institute of Science and Technology, Daejeon  34141, Republic of Korea}
\email{lcuh11@kaist.ac.kr}



\thanks{\textbf{Acknowledgment.} 
The authors thank Professor Moon-Jin Kang for the beneficial discussions.
The work of X. Huang was supported by the National Research Foundation of Korea (NRF) grant funded by the Korea government (MSIT) (No. RS-2019-NR040050). H. Lee was partially supported by the National Research Foundation of Korea (RS-2024-00361663 and NRF-2019R1A5A1028324)}

\begin{document}

\begin{abstract} 
This work focuses on the long-time nonlinear stability of a composite wave pattern comprising two viscous shocks and a viscous contact wave for the one-dimensional compressible Navier-Stokes-Fourier (NSF) system. Specifically, we establish that if the composite wave strength and the perturbations are sufficiently small, the NSF system admits a unique global-in-time strong solution, which converges uniformly in space as time tends to infinity, towards the corresponding composite wave, up to dynamical shifts in the positions of the two viscous shocks. Notably, the strengths of the two viscous shocks can be chosen independently. Our proof relies upon the $a$-contraction method with time-dependent shifts and suitable weight functions.
\end{abstract}

\maketitle
	\tableofcontents
\section{Introduction}
We study the one-dimensional compressible Navier–Stokes–Fourier (NSF) system formulated in Lagrangian coordinates
\begin{equation}\label{eq:NSF}
	\left\{\begin{aligned}
		&v_t-u_x=0, \quad  && x\in \mathbb{R},\quad  t>0 \\
		&u_t+p(v,\theta)_x=\left(\mu\frac{u_x}{v}\right)_x,   \\
		& E_t + \left(p(v,\theta)u\right)_x= \left(\kappa \frac{\theta_x}{v}\right)_x + \left(\mu\frac{uu_x}{v}\right)_x,
	\end{aligned}\right.
\end{equation}
In the model, $v = v(t,x)$, $u = u(t,x)$, and $\theta = \theta(t,x)$ denote the fluid’s specific volume, velocity, and absolute temperature, respectively. The total energy is defined by $E = e + \frac{u^2}{2}$, where $e$ represents the internal energy.
For an ideal polytropic gas, the pressure $p$ and internal energy $e$ are given by
$$p(v, \theta) = \frac{R\theta}{v},\qquad e(v, \theta) = \frac{R}{\gamma -1} \theta + constant,$$
where $R > 0$ and $\gamma > 1$ are physical constants characteristic of the gas.
The constants $\mu$ and $\kappa$ denote the viscosity and the heat-conductivity coefficients, respectively.

The system \eqref{eq:NSF} is supplied with initial data $(v_0, u_0, \theta_0)$ connecting the prescribed far-field constant states:
\begin{align}\label{eq:farfield}
	\lim_{x\to \pm{\infty}} (v_0(x),u_0(x),\theta_0(x)) = (v_{\pm}, u_{\pm}, \theta_{\pm}).
\end{align}

It is well understood that the asymptotic behavior of solutions 
$(v, u, \theta)$ to the NSF equations \eqref{eq:NSF} as $t\rightarrow \infty$ is fundamentally governed by the Riemann problem of the corresponding full compressible Euler equations:

\begin{equation}\label{eq:CE}
	\left\{\begin{aligned}
		&v_t-u_x=0, \quad  && x\in \mathbb{R}, \quad t>0, \\
		&u_t+p(v,\theta)_x=0,   \\
		& E_t + \left(p(v,\theta)u\right)_x= 0,
	\end{aligned}\right.
\end{equation}
with the Riemann-type initial data
\begin{equation}\label{CEinitial}
	(v(0,x), u(0,x), \theta(0,x)) =
	\left\{\begin{array}{ll}
		(v_-, u_-,\theta_-), &\di x<0,\\
		(v_+, u_+,\theta_+), &\di x>0.
	\end{array}\right.
\end{equation}
The system \eqref{eq:CE} is a strictly hyperbolic conservation law with three distinct real eigenvalues
$$\lambda_1  = -\sqrt{\frac{\gamma p}{v}} < 0, \quad \lambda_2 = 0, \quad \lambda_3 =-\lambda_1> 0.$$
The characteristic fields associated with $\lambda_1$ and $\lambda_3$ are genuinely nonlinear and thus give rise to either shock wave or rarefaction wave, whereas the field corresponding to $\lambda_2$ is linearly degenerate and therefore generates a contact discontinuity. Consequently, the self-similar solution of the Riemann problem, commonly referred to as the Riemann solution, is composed of these three elementary wave patterns and their admissible combinations (see, e.g., \cite{D79}).
The viscous counterparts of these Riemann waves subsequently determine the large-time behavior of solutions to the Cauchy problem \eqref{eq:NSF}–\eqref{eq:farfield}.


We are interested in the case where the end states $(v_{\pm}, u_{\pm}, \theta_{\pm}) \in \mathbb{R}_+ \times \mathbb{R} \times \mathbb{R}_+$ of the Riemann problem consist of two shocks and a contact discontinuity. 
Specifically, there exist unique states $(v_*, u_*,\theta_*)$ and $(v^*, u^*,\theta^*) \in \mathbb{R}_+ \times \mathbb{R}\times \mathbb{R}_+$ such that the left state $(v_-, u_-, \theta_-)$ is first connected to $(v_*, u_*,\theta_*)$ by a $1$-shock. 
From $(v_*, u_*,\theta_*)$, one then moves along the second contact discontinuity curve to reach $(v^*, u^*,\theta^*) \in CD_2(v^*, u^*,\theta^*)$. 
Finally, the state $(v^*, u^*,\theta^*)$ is connected to the right state $(v_+, u_+, \theta_+)$ through a $3$-shock.

If $(v_l, u_l, \theta_l) \in S_1(v_r, u_r, \theta_r) \cup S_3(v_r, u_r, \theta_r)$, 
then the left and right states are connected by a shock satisfying the Rankine--Hugoniot conditions
\begin{align}
	\begin{aligned}
		& -\sigma (v_r- v_l) - (u_r -u_l) =0,  \\
		& -\sigma (u_r -u_l) - (p(v_r, \theta_r) - p(v_l, \theta_l)) =0, \\
		& -\sigma (E_r -E_l) + (p(v_r, \theta_r)u_r - p(v_l, \theta_l)u_l)=0,
	\end{aligned}
\end{align}
which define two admissible shock curves in the state space.
More precisely, the $1$-shock curve $S_1(v_r, u_r, \theta_r)$ corresponds to $v_l>v_r$ with
\[
\sigma=-\sqrt{-\frac{p(v_r, \theta_r)-p(v_l, \theta_l)}{v_r-v_l}},
\]
while the admissible $3$-shock curve corresponds to $v_l<v_r$ with
\[
\sigma=\sqrt{-\frac{p(v_r, \theta_r)-p(v_l, \theta_l)}{v_r-v_l}}.
\]

In either case, the associated shock solution $(v^s,u^s,\theta^s)$ to \eqref{eq:CE}--\eqref{CEinitial}
connecting $(v_l,u_l,\theta_l)$ and $(v_r,u_r,\theta_r)$ is given by the discontinuous traveling wave
\begin{equation}\label{shocksolution}
	(v^s, u^s, \theta^s)(t,x)=
	\left\{\begin{array}{ll}
		(v_l, u_l,\theta_l), & x<\sigma t,\\
		(v_r, u_r,\theta_r), & x>\sigma t.
	\end{array}\right.
\end{equation}

The second characteristic field, associated with the eigenvalue $\lambda_2=0$, is linearly degenerate and gives rise to a $2$-contact discontinuity.
For a given right state $(v_r, u_r, \theta_r) \in \mathbb{R}_+ \times \mathbb{R} \times \mathbb{R}_+$, 
if the left state $(v_l, u_l, \theta_l)$ lies on the $2$-contact discontinuity curve
\begin{align}
	CD_2(v_r, u_r, \theta_r)
	:= \{(v,u,\theta)\mid u_l=u_r,\ p(v_l,\theta_l)=p(v_r,\theta_r)\},
\end{align}
then the unique contact discontinuity connecting
$(v_l, u_l, \theta_l)$ and $(v_r, u_r, \theta_r)$
is given by the piecewise constant solution
\begin{align}\label{CD}
	(v^d, u^d, \theta^d)(t,x)=
	\begin{cases}
		(v_l, u_l, \theta_l), & x<u_r t, \\
		(v_r, u_r, \theta_r), & x>u_r t,
	\end{cases}
\end{align}
which solves the inviscid system \eqref{eq:CE}--\eqref{CEinitial} with Riemann initial data.

The Riemann problem \eqref{eq:CE}-\eqref{CEinitial} admits a unique Riemann solution $(\hat{v}, \hat{u}, \hat{\theta})$, which is constructed by piecing together three elementary waves associated with the system’s characteristic fields.
\begin{align*}
	(\hat{v}, \hat{u}, \hat{\theta})= (v^s_1,u^s_1,\theta^s_1)(t,x) + (v^d,u^d,\theta^d)(t,x) + (v^s_3,u^s_3,\theta^s_3)(t,x) - (v_*, u_*,\theta_*) - (v^*, u^*,\theta^*).
\end{align*}

The viscous counterpart of the Riemann solution $(\bar{v},\bar{u},\bar{\theta})$ is given by the composite wave:
\begin{align}
\begin{aligned}\label{composite wave}
	(\bar{v},\bar{u},\bar{\theta})=& (\tilde{v}_1,\tilde{u}_1,\tilde{\theta}_1)(x-\sigma_1t) + (v^D,u^D,\theta^D)(t,x) \\
    &+ (\tilde{v}_3,\tilde{u}_3,\tilde{\theta}_3)(x-\sigma_3t)- (v_*, u_*,\theta_*) - (v^*, u^*,\theta^*),
\end{aligned}
\end{align}
which is composed of a $1$-viscous shock $(\tilde{v}_1,\tilde{u}_1,\tilde{\theta}_1)(x-\sigma_1t)$, a $3$-viscous shock $(\tilde{v}_3,\tilde{u}_3,\tilde{\theta}_3)(x-\sigma_3t)$, and the viscous contact wave $(v^D,u^D,\theta^D)(t,x)$. 

For each $i = 1,3$, we denote by $(\tilde{v}_i, \tilde{u}_i, \tilde{\theta}_i)(x - \sigma_i t)$ the viscous shock wave, which is a traveling wave solution of \eqref{eq:NSF} satisfying
\begin{equation}\label{eq:VS}
	\begin{cases}
		&-\sigma_i(\tilde{v}_i)'-(\tilde{u}_i)'=0,\\
		&-\sigma_i(\tilde{u}_i)'+(\tilde{p}_i)'=\left(\mu \frac{(\tilde{u}_i)'}{\tilde{v}_i}\right)', \\
		& -\sigma_i(\tilde{E}_i)' + (\tilde{p}_i\tilde{u}_i)' = \left(\kappa\frac{(\tilde{\theta}_i)'}{\tilde{v}_i}\right)' + \left(\mu \frac{\tilde{u}_i(\tilde{u}_i)'}{\tilde{v}_i}\right)', \\
		& (\tilde{v}_1, \tilde{u}_1, \tilde{\theta}_1)(-\infty)=(v_-, u_-,\theta_-), \quad (\tilde{v}_1, \tilde{u}_1, \tilde{\theta}_1)(+\infty)=(v_*, u_*,\theta_*),  \\
		& (\tilde{v}_3, \tilde{u}_3, \tilde{\theta}_3)(-\infty)=(v^*, u^*,\theta^*), \quad (\tilde{v}_3, \tilde{u}_3, \tilde{\theta}_3)(+\infty)=(v_+, u_+,\theta_+), 
	\end{cases}
\end{equation}
where $\tilde{E}_i:= \frac{R}{\gamma -1} \tilde{\theta}_i + \frac{(\tilde{u}_i)^2}{2}$ and $\tilde{p}_i := p(\tilde{v}_i,\tilde{\theta}_i)$.

The viscous analogue of the inviscid contact discontinuity connecting
$(v_*, u_*, \theta_*)$ and $(v^*, u^*, \theta^*)$
is the viscous contact wave $(v^D, u^D, \theta^D)(t,x)$.
Following \cite{HXY08}, it is given by
\begin{align}\label{VCD}
	\begin{aligned}
		v^D(t,x)&= \frac{R\Theta^{sim}}{p_*}, \\
		u^D(t,x)&= u_* + \frac{(\gamma-1)\kappa\,\Theta_x^{sim}}{R\gamma \Theta^{sim}}, \\
		\theta^D(t,x)&= \Theta^{sim},
	\end{aligned}
\end{align}
where $\Theta^{sim}=\Theta^{sim}\!\left(\frac{x}{\sqrt{1+t}}\right)$
is the unique self-similar solution of
\begin{equation*}
	\begin{cases}
		\Theta_t = \dfrac{(\gamma-1)\kappa p_*}{R^2\gamma}
		\left(\dfrac{\Theta_x}{\Theta}\right)_x, \\[0.3em]
		\Theta(t,-\infty)=\theta_*, \qquad \Theta(t,+\infty)=\theta^*.
	\end{cases}
\end{equation*}

The long-time asymptotic behavior of solutions to the Navier–Stokes (NS) or NSF systems is fundamentally characterized by three elementary viscous wave patterns: viscous shock wave, viscous contact wave, and rarefaction wave.
The stability theory for viscous shocks has been developed through a substantial body of work, ranging from early foundational analyses to later advances addressing zero-mass constraints, spectral conditions, and degenerate or generalized viscosity mechanisms \cite{MN85, G86, Liu85, SX93, LZ15, MZ04, MW10, VY16, BD06}.
For the linearly degenerate characteristic field, while inviscid contact discontinuities are known to be unstable, their viscous counterparts admit diffusive profiles whose asymptotic stability has been rigorously established for both artificially regularized systems and the physical NSF equations \cite{X96, LX97, HXY08, HLM10}.
The stability of rarefaction waves has likewise been well documented, beginning with classical results for the NS equations and subsequently extended to the full NSF system \cite{MN86, MN92, LX88, NYZ04}.
Collectively, these contributions provide a comprehensive and robust theoretical framework for the stability of the three fundamental viscous wave patterns in compressible fluid dynamics.

 Building upon these foundational results, research attention has gradually shifted toward the stability of composite wave patterns, where several elementary waves interact dynamically. While the aforementioned works focus on individual viscous shock, rarefaction wave, or contact discontinuity, studies on their superposition are relatively scarce.

 In this direction, Huang and Matsumura \cite{HM09} studied the asymptotic stability of superposed viscous shock waves for the NSF system, while Huang, Li, and Matsumura \cite{HLM10} established the stability of composite waves consisting of a viscous contact wave and rarefaction waves.
 The interaction between a viscous shock and a rarefaction wave is more delicate \cite{M18, KVW23, KVW-NSF}, since the anti-derivative method for viscous shocks is incompatible with the energy method for rarefaction waves.
 This difficulty was overcome by Kang, Vasseur, and Wang \cite{KVW23}, who introduced the $a$-contraction method with shift to prove the stability of a viscous shock–rarefaction composite wave for the NS equations.
 Their approach was later extended in \cite{KVW-NSF} to the full NSF system, yielding the stability of general Riemann solutions composed of a viscous shock, a rarefaction wave, and a viscous contact wave.
 More recently, Han, Kang, and Kim \cite{HKK23} applied the $a$-contraction method to configurations involving two viscous shocks with independent small amplitudes, thereby removing the comparability assumption required in \cite{HM09}.


 In this paper, we establish the nonlinear time-asymptotic stability of the composite wave
 $(\bar{v}, \bar{u}, \bar{\theta})$ for the one-dimensional NSF equations
 \eqref{eq:NSF}--\eqref{eq:farfield},
 which consists of a shifted $1$-viscous shock, a $2$-viscous contact discontinuity,
 and a shifted $3$-viscous shock,
 under smallness assumptions on both the wave strength and the perturbation.
\\

\subsection{Main result}
We state the main theorem on the global existence and asymptotic behavior of the Cauchy problem for the NSF equations.

\begin{theorem}\label{main theorem}
	Let $(v_+,u_+,\theta_+) \in \mathbb{R}_+ \times \mathbb{R} \times \mathbb{R}_+$ be a given constant state. Then there exist positive constants $\delta_0$ and $\varepsilon_0$ such that the following assertions hold.
	
	Assume that the states $(v_-,u_-,\theta_-)$, $(v_*,u_*,\theta_*)$, and $(v^*,u^*,\theta^*)$ satisfy
	\[(v_-,u_-,\theta_-) \in S_1(v_*,u_*,\theta_*),\quad  (v_*,u_*,\theta_*)\in CD_2(v^*,u^*,\theta^*), \quad  (v^*,u^*,\theta^*) \in S_3(v_+,u_+,\theta_+),\]
	and are sufficiently close in the sense that
		\[|v_--v_\ast|+ |v_*-v^*|+|v^*-v_+| \leq \delta_0.\]
Let $(\widetilde{v}_1, \widetilde{u}_1, \widetilde{\theta}_1)(x-\sigma_1 t)$ and $(\widetilde{v}_3, \widetilde{u}_3, \widetilde{\theta}_3)(x-\sigma_3 t)$ denote the $1$- and $3$-viscous shock profiles of \eqref{eq:VS} connecting $(v_-,u_-,\theta_-)$ to $(v_*,u_*,\theta_*)$ and $(v^*,u^*,\theta^*)$ to $(v_+,u_+,\theta_+)$, respectively.
Let $(v^D,u^D,\theta^D)(t,x)$ be the $2$-viscous contact wave defined by \eqref{VCD} joining $(v_*,u_*,\theta_*)$ and $(v^*,u^*,\theta^*)$.
		
Suppose that the initial data $(v_0,u_0,\theta_0)$ satisfy
	\[\sum_{\pm}\left(\|(v_0-v_{\pm}, u_0- u_{\pm}, \theta_0-\theta_{\pm}\|_{L^2(\mathbb{R_{\pm}})}\right)
+ \|(v_{0x}, u_{0x}, \theta_{0x})\|_{L^2(\mathbb{R})} \leq \varepsilon_0,\]
where $\mathbb{R_-} =:- \mathbb{R_+} =(-\infty, 0)$.

Then the compressible Navier–Stokes–Fourier system \eqref{eq:NSFS} (equivalently, \eqref{eq:NSF}) admits a unique global-in-time solution $(v,u,\theta)(t,x)$ for all $t > 0$.
Moreover, there exist absolutely continuous shift functions $\bold{X}_1(t)$ and $\bold{X}_3(t)$ such that
	\begin{align}
	\begin{aligned}
		& v(t,x) - \left(\tilde{v}_1(x-\sigma_1t -\bold{X}_1(t)) + v^D(\frac{x}{\sqrt{1+t}}) +\tilde{v}_3(x-\sigma_3t -\bold{X}_3(t)) -v_* -v^*\right) \in C\left(0,+\infty; H^1(\mathbb{R})\right), \\
		& u(t,x) - \left(\tilde{u}_1(x-\sigma_1t -\bold{X}_1(t)) + u^D(\frac{x}{\sqrt{1+t}}) +\tilde{u}_3(x-\sigma_3t -\bold{X}_3(t)) -u_* -u^*\right) \in C\left(0,+\infty; H^1(\mathbb{R})\right), \\
		& \theta(t,x) - \left(\tilde{\theta}_1(x-\sigma_1t -\bold{X}_1(t)) + \theta^D(\frac{x}{\sqrt{1+t}}) +\tilde{\theta}_3(x-\sigma_3t -\bold{X}_3(t))-\theta_* -\theta^* \right) \in C\left(0,+\infty; H^1(\mathbb{R})\right).
	\end{aligned}
\end{align}
In addition, the solution satisfies the large-time asymptotic behavior
	\begin{align}\label{est:longtime}
	\begin{aligned}
		\sup_{x\in \mathbb{R}}\bigg|(v,u,\theta)(t,x) - \bigg(& \tilde{v}_1(x-\sigma_1t -\bold{X}_1(t)) + v^D(\frac{x}{\sqrt{1+t}}) +\tilde{v}_3(x-\sigma_3t -\bold{X}_3(t)) -v_* -v^*,  \\
		& \tilde{u}_1(x-\sigma_1t -\bold{X}_1(t)) + u^D(\frac{x}{\sqrt{1+t}}) +\tilde{u}_3(x-\sigma_3t -\bold{X}_3(t)) -u_* -u^*, \\
		& \tilde{\theta}_1(x-\sigma_1t -\bold{X}_1(t)) + \theta^D(\frac{x}{\sqrt{1+t}}) +\tilde{\theta}_3(x-\sigma_3t -\bold{X}_3(t))-\theta_* -\theta^*\bigg)\bigg| \to 0,
	\end{aligned}
\end{align}
as $t \to \infty$.

Moreover, the shifts satisfy
	\begin{align}\label{shiftproperty}
	\lim_{t\to +\infty} |\dot{\bold{X}}_i(t)| =0,  \quad i =1,3.
\end{align}
together with the wave-separation property
	\begin{align}\label{est:waveseparation}
	\bold{X}_1 (t)+ \sigma_1t  \leq \frac{\sigma_1}{2}t <0 < \frac{\sigma_3}{2}t\le\bold{X}_3(t) +\sigma_3t , \quad t>0.
\end{align}
\end{theorem}

\begin{remark}
	Theorem \ref{main theorem} asserts that when the two far-field states $(v_{\pm}, u_{\pm}, \theta_{\pm})$ specified in \eqref{eq:farfield} are connected through a composite Riemann structure consisting of two shocks and a contact discontinuity, the corresponding solution of the compressible Navier–Stokes–Fourier system \eqref{eq:NSF} (or equivalently, \eqref{eq:NSFS}) exists globally and asymptotically approaches, as $t \to +\infty$, a composite viscous profile. This limiting profile is formed by the superposition of the two viscous shock waves, each modulated by a time-dependent shift $\bold{X}_i(t)$, $i=1,3$, and the viscous contact wave that bridges them.
\end{remark}

\begin{remark}
The shift functions $\bold{X}_i(t)$ $($defined in \eqref{def:shift}$)$ satisfy the asymptotic property stated in \eqref{shiftproperty}, from which it follows that
	\[\lim_{t\to \infty} \frac{\bold{X}_i(t)}{t} =0, \quad  \ i=1,3. \]
This relation indicates that the shifts evolve at most sublinearly in time, implying that the corresponding viscous shock profiles remain essentially stationary in shape when viewed in their comoving coordinates. Consequently, the long-time dynamics preserve the traveling wave structure of each viscous shock.
\end{remark}

\subsection{Main ideas}
The central analytical tool employed to establish the nonlinear stability result is the $a$-contraction method with shifts. 
This technique, introduced in \cite{KV16, LV11}, provides a robust framework for studying the stability of shock waves in inviscid conservation laws, including the Euler system \eqref{eq:CE}.
It was subsequently extended to viscous models; see, for instance,
\cite{KV17, K21, KVJEMS, KV22, KVW23, KVW-NSF}. Based on the relative entropy formulation \cite{D96}, the method seeks a suitable weight function $a$ and a time-dependent shift $\bold{X}(t)$ such that the weighted relative entropy is non-increasing in time:
\[
\frac{d}{dt}\int_{\mathbb R}
a(t,x-\bold{X}(t))\,\eta\!\left(U(t,x)\mid \bar U(x-\bold{X}(t))\right)\,dx \le 0.
\]
This inequality quantifies the contraction of solutions toward the reference shock profile.

As discussed in the introduction (see, for example, \cite{ KVW-NSF}), the $a$-contraction framework can be extended beyond single wave to establish the stability of generic composite wave patterns consisting of a viscous shock, a viscous contact wave, and a rarefaction wave. The present study consider the configuration described in \eqref{composite wave}, where two viscous shocks and one viscous contact wave are superposed. This three-wave composition introduces several new analytical difficulties. Unlike the case of a single viscous shock, the superposition of two viscous shocks and a viscous contact wave does not form an exact solution of the NSF system because of nonlinear coupling among the individual profiles. Consequently, one must carefully estimate the nonlinear interaction terms generated by the mutual influence of the $1$-shock, $2$-contact wave, and $3$-shock components. To manage these interactions, we employ two independent shift functions, $\bold{X}_1(t)$ and $\bold{X}_3(t)$, defined in \eqref{def:shift}, together with corresponding weight functions introduced in \eqref{def_a}. These functions allow us to localize perturbations near each viscous shock and to apply the single-wave $a$-contraction argument in a localized fashion. In addition, to control error terms that are spatially confined between the two shock layers, we introduce auxiliary cutoff functions, defined in terms of the shifts as in \eqref{def:cutoff}. The construction of these localization functions ensures that the two shock positions remain sufficiently separated over time, as guaranteed by the separation property \eqref{est:waveseparation}.\\

The paper is organized as follows.
Section \ref{preliminaries} collects preliminary results, including basic properties of viscous shock and contact waves and a Poincaré-type inequality.
Section \ref{a-priori} presents the analytical framework, including local well-posedness, the construction of shift functions, and the main a priori estimate in Proposition \ref{prop:main}, from which Theorem \ref{main theorem} follows via a continuation argument.
Sections \ref{relative estimate} and \ref{higher order} are devoted to the proof of Proposition \ref{prop:main}, where Section \ref{relative estimate} develops the $a$-contraction method and Section \ref{higher order} establishes the higher-order energy estimates.

\section{Preliminaries}\label{preliminaries}
\subsection{Viscous shock wave}

The following lemma summarizes basic properties of viscous shock profiles; see \cite{KVW-NSF, EEKO24} for details.


\begin{lemma}\label{lem:vs}
Let $(v_r,u_r,\theta_r)$ be a given right-end state. Then there exists a constant $C>0$ such that the following holds.
For any left-end state $(v_l,u_l,\theta_l)$ that is connected to $(v_r,u_r,\theta_r)$ through an $i$-shock curve $(i=1,3)$, there exists a unique traveling-wave solution $(\tilde{v}_i(x-\sigma_i t), \tilde{u}_i(x-\sigma_i t), \tilde{\theta}_i(x-\sigma_i t))$ of \eqref{eq:VS}.

Denote by $\delta_i$ the strength of the $i$-shock, defined by
 $\delta_i:=|u_r-u_l|\sim|v_r-v_l| \sim |\theta_r -\theta_l|$. 
 Then the solution $(\tilde{v}_i, \tilde{u}_i, \tilde{\theta}_i)$ satisfies the following properties:
	\begin{align}
		\begin{aligned}\label{est:shock_prop}
			&\widetilde{v}_1 '(x-\sigma_1t)<0, \quad \widetilde{u}_1 '(x-\sigma_1t)<0, \quad \widetilde{\theta}_1 '(x-\sigma_1t)>0, \quad \forall (x-\sigma_1t) \in \mathbb{R},  \\
			&\widetilde{v}_3 '(x-\sigma_3t)>0, \quad \widetilde{u}_3 '(x-\sigma_3t)<0, \quad \widetilde{\theta}_3 '(x-\sigma_3t)<0, \quad \forall  (x-\sigma_3t) \in \mathbb{R},  \\
			&|(\tv_i(x-\sigma_it)-v_{L}, \tu_i(x-\sigma_it)-u_l, \tilde{\theta}_i(x-\sigma_it)- \theta_l)|\le C\delta_i e^{-C\delta|x-\sigma_it|},\quad  x-\sigma_it<0,\\
			& |(\tv_i(x-\sigma_it)-v_{R}, \tu_i(x-\sigma_it)-u_r, \tilde{\theta}_i(x-\sigma_it)- \theta_r)|\le C\delta_i e^{-C\delta|x-\sigma_it|},\quad  x-\sigma_it>0, \\
			&|(\tv_i'(x-\sigma_it),\tu_i'(x-\sigma_it), \tilde{\theta}_i'(x-\sigma_it))|\le C\delta_i^2 e^{-C\delta_i|x-\sigma_it|},\quad \forall(x-\sigma_it)\in\R,\\
		& |(\tv_i''(x-\sigma_it),\tu_i''(x-\sigma_it),\tilde{\theta}_i''(x-\sigma_it))| \le C\delta_i|(\tv_i'(x-\sigma_it),\tu_i'(x-\sigma_it), \tilde{\theta}_i'(x-\sigma_it))|,\quad \forall(x-\sigma_it) \in\R.
		\end{aligned}
	\end{align}

Additionally, $|\tilde{v}_i'| \sim |\tilde{u}_i'| \sim |\tilde{\theta}_i'|$ for all $(x-\sigma_it) \in \mathbb{R}$, which mean exactly
		\begin{align} 
			\begin{aligned} \label{est:vu}
				|\tilde{u}_1' - \sigma_-\tilde{v}_1'| &\leq C\delta_1 |\tilde{v}_1'|, \quad \forall (x-\sigma_1t)\in \mathbb{R}, \\
				|\tilde{u}_3' + \sigma^\ast\tilde{v}_3'| &\leq C\delta_3 |\tilde{v}_3'|, \quad \forall (x-\sigma_3t)\in \mathbb{R},\\
			\end{aligned}
		\end{align}
		and 
		\begin{align}
			\begin{aligned} \label{est:vtheta}
				|\tilde{\theta}_1' + \frac{(\gamma -1)p_-}{R}\tilde{v}_1'| &\leq C\delta_1 |\tilde{v}_1'|, \quad \forall (x-\sigma_1t) \in \mathbb{R},\\
				|\tilde{\theta}_3' + \frac{(\gamma -1)p^\ast}{R}\tilde{v}_3'| &\leq C\delta_3 |\tilde{v}_3'|, \quad \forall (x-\sigma_3t) \in \mathbb{R},\\
			\end{aligned}
		\end{align}
		where 
		\begin{align*}
			&p_-:= p(v_-,\theta_-)= \frac{R\theta_-}{v_-}, \ p^\ast:= p(v^\ast,\theta^\ast)= \frac{R\theta^\ast}{v^\ast}, \ \text{and}\\
			&\sigma_-:= \sqrt{\frac{\gamma p_-}{v_-}}, \ \sigma^\ast:= \sqrt{\frac{\gamma p^\ast}{v^\ast}}.\\
		\end{align*}
Also, $\sigma_-$ and $\sigma^\ast$ above satisfy
		\begin{align}\label{est:sigma}
			|\sigma_1+\sigma_-|\leq C\delta_1 \ \text{\and} \ |\sigma_3 -\sigma^\ast|\leq C\delta_3.
		\end{align}     
\end{lemma} 
\begin{remark}
	In this paper, we adopt the notation 
	 $(v_l,u_l,\theta_l)= (v_-,u_-,\theta_-)$ and $(v_r,u_r,\theta_r)= (v_*,u_*,\theta_*)$ for the $1$-viscous shock, and $(v_l,u_l,\theta_l)= (v^*,u^*,\theta^*)$ and $(v_r,u_r,\theta_r)= (v_+,u_+,\theta_+)$ for the $3$-viscous shock in the subsequent analysis.
\end{remark}

\subsection{Viscous contact wave}

Like the viscous shock, there is a corresponding viscous part, called `viscous contact wave'. It can be explicitly constructed and has been proven to be time-asymptotically stable, both in the context of the artificial viscosity model \cite{LX97, X96} and the physical compressible Navier–Stokes equations \cite{HXY08, HLM10}.
\begin{lemma}\label{lem:vcd}
	$($\cite{HXY08}$)$ The viscous contact wave $(v^D,u^D,\theta^D)(t,x)$ defined in \eqref{VCD} satisfies
	\begin{align}\label{est:vcdproperty}
		\begin{aligned}
			& \left(v^D -v_*, u^D-u_*, \theta^D- \theta_*\right) = O(1) \delta_Ce^{-\frac{C_1x^2}{1+t}}, \quad \forall x<0; \\
			&  \left(v^D -v^*, u^D-u^*, \theta^D- \theta^*\right) = O(1) \delta_Ce^{-\frac{C_1x^2}{1+t}}, \quad \forall x>0;   \\
			& \left(\partial_x^nv^D, \partial_x^n\theta^D\right)(t,x) = O(1)\delta_C(1+t)^{-\frac{n}{2}}e^{-\frac{C_1x^2}{1+t}}, \quad \forall x\in \mathbb{R}, \quad n=1,2,\dots;\\
			& \partial_x^nu^D(t,x) = O(1)\delta_C(1+t)^{-\frac{1+n}{2}}e^{-\frac{C_1x^2}{1+t}}, \quad \forall x\in \mathbb{R}, \quad n=1,2,\dots;
		\end{aligned}
	\end{align}
	where $\delta_C = |v^* -v_*| \sim |u^*-u_*| \sim|\theta^* -\theta_*|$ is the amplitude of the viscous contact wave, and $C_1>0$ is generic constant.
\end{lemma}

The viscous wave $(v^D, u^D, \theta^D)(t,x)$ defined in \eqref{VCD} satisfies the system
\begin{equation}\label{eq:VCD}
	\begin{cases}
		v^D_t - u^D_x = 0, \\
		u^D_t + p^C_x = \mu \left(\frac{u^D_x}{v^D}\right)_x + Q_1^C, \\
		\frac{R}{\gamma - 1} \theta^D_t + p^C u^D_x = \kappa \left(\frac{\theta^D_x}{v^D}\right)_x + \mu \frac{(u^D_x)^2}{v^D} + Q_2^C,
	\end{cases}
\end{equation}
where $p^C = p(v^D, \theta^D)$ and the remainder terms satisfy
\begin{align}\label{est:QC}
	\begin{aligned}
		Q_1^C &= u^D_t - \mu \left(\frac{u^D_x}{v^D}\right)_x = O(1) , \delta_C (1+t)^{-\frac{3}{2}} e^{-\frac{2 C_1 x^2}{1+t}}, \\
		Q_2^C &= - \mu \frac{(u^D_x)^2}{v^D} = O(1) \delta_C (1+t)^{-2} e^{-\frac{2 C_1 x^2}{1+t}},
	\end{aligned}
\end{align}
as $|x| \to \infty$, according to Lemma \ref{lem:vcd}. Furthermore, from \eqref{VCD} and Lemma \ref{lem:vcd}, for any $p \geq 1$ it holds that
\begin{align*}
	\|(v^D,u^D,\theta^D)(t,\cdot)- (v^d,u^d,\theta^d)(t,\cdot)\|_{L^p(\mathbb{R})} = O(1)\kappa^{\frac{1}{2p}}(1+t)^{\frac{1}{2p}}, 
\end{align*}
In any finite time horizon, the viscous contact wave $(v^D, u^D, \theta^D)$ approaches the inviscid contact discontinuity $(v^d, u^d, \theta^d)$ in the $L^p$ sense as the heat conductivity $\kappa$ tends to zero. Nevertheless, this convergence does not persist uniformly for large times, where noticeable discrepancies between the two profiles may develop.

\subsection{Poincar\'e-type inequality}
We utilize effectively the following Poincar\'e-type inequality in the method of $a$-contraction with shift in the viscous cases.

\begin{lemma}[\cite{KV21}, Lemma 2.9]\label{Poincare}
	For any $f:[0,1]\rightarrow\mathbb{R}$ satisfying $\int_{0}^{1}y(1-y)|f'| dy <\infty$
	\[\int_{0}^1 \left|f-\int_{0}^1f dy\right|^2dy\le \frac{1}{2}\int_{0}^1 y(1-y)|f'|^2dy.\]
\end{lemma}

\section{A priori estimate and proof of the main theorem}\label{a-priori}
\setcounter{equation}{0}
In this section, we derive a priori $H^1$-estimates for the perturbation of the solution around the composite wave.
These estimates form the basis for establishing the large-time convergence of the solution toward the composite wave profile.

For later convenience, we rewrite the system \eqref{eq:NSF} in a non-divergence form equivalent to the original one.
\begin{equation} \label{eq:NSFS}
	\begin{cases}
		&v_t-u_x=0, \\
		&u_t+p(v,\theta)_x=\left(\mu\frac{u_x}{v}\right)_x,    \\
		&\frac{R}{\gamma-1}\theta_t+p(v,\theta)u_x= \left(\kappa \frac{\theta_x}{v}\right)_x + \mu \frac{u_x^2}{v}.
	\end{cases}
\end{equation}

We will analyze the stability of solutions to \eqref{eq:NSFS} around a superposition wave, which consists of two viscous shock waves shifted by $\bold{X}_1(t) $ and $\bold{X}_3(t)$  (as defined in \eqref{def:shift}), along with a viscous contact wave:

\begin{equation*} 
	\begin{pmatrix}
		\bar{v}(t,x)\\
		\bar{u}(t,x)\\
		\bar{\theta}(t,x)
	\end{pmatrix}:=
	\begin{pmatrix}
		\widetilde{v}_1(x-\sigma_1t -\bold{X}_1(t)) + v^{C}(t,x) + \widetilde{v}_3(x-\sigma_3t -\bold{X}_3(t))-v_* -v^*\\
		\widetilde{u}_1(x-\sigma_1t -\bold{X}_1(t)) + u^{C}(t,x) + \widetilde{u}_3(x-\sigma_3t -\bold{X}_3(t))-u_*-u^*\\
		\widetilde{\theta}_1(x-\sigma_1t -\bold{X}_1(t)) + \theta^{C}(t,x) + \widetilde{\theta}_3(x-\sigma_3t -\bold{X}_3(t))-\theta_* -\theta^*\\
	\end{pmatrix}.
\end{equation*}
The composite wave $(\bar{v}, \bar{u}, \bar{\theta})(t,x)$ satisfies the system
\begin{equation}\label{eq:NSFCW}
	\begin{cases}
		&\bar{v}_t +\dot{\bold{X}}_1(t)(\widetilde{v}^{\bold{X}_1}_1)_x+\dot{\bold{X}}_3(t)(\widetilde{v}^{\bold{X}_3}_3)_x-\bar{u}_x=0\\
		&\bar{u}_t+\dot{\bold{X}}_1(t)(\widetilde{u}^{\bold{X}_1}_1)_x+\dot{\bold{X}}_3(t)(\widetilde{u}^{\bold{X}_3}_3)_x+\bar{p}_x =\left(\mu \frac{\bar{u}_x}{\bar{v}}\right)_x+Q_1,\\
		& \frac{R}{\gamma-1}\bar{\theta}_t +\frac{R}{\gamma-1}\dot{\bold{X}}_1(t)(\widetilde{\theta}^{\bold{X}_1}_1)_x +\frac{R}{\gamma-1}\dot{\bold{X}}_3(t)(\widetilde{\theta}^{\bold{X}_3}_3)_x +\bar{p}\bar{u}_x= \left(\kappa \frac{\bar{\theta}_x}{\bar{v}}\right)_x + \mu \frac{\bar{u}^2_x}{\bar{v}}+Q_2,
	\end{cases}
\end{equation}
where $\bar{p}={R\bar{\theta}\over \bar{v}}$ and the error terms 
\begin{equation*}
	Q_i:= Q^I_i + Q^C_i , \quad  i=1, 2,
\end{equation*}
with the wave interaction terms
\begin{align}\label{def:Q}
	\begin{aligned}
		Q^I_1:= &(\bar{p}- p^{\bold{X}_1}_1 - p^C - p^{\bold{X}_3}_3)_x - \mu \left(\frac{\bar{u}_x}{\bar{v}} - \frac{(\widetilde{u}^{\bold{X}_1}_1)_x}{\widetilde{v}^{\bold{X}_1}_1} - \frac{u^D_x}{v^D} - \frac{(\widetilde{u}^{\bold{X}_3}_3)_x}{\widetilde{v}^{\bold{X}_3}_3}\right)_x,   	\\
		Q^I_2 := & \left(\bar{p}\bar{u}_x - p^{\bold{X}_1}_1(\widetilde{u}^{\bold{X}_1}_1)_x - p^Cu^D_x - p^{\bold{X}_3}_3(\widetilde{u}^{\bold{X}_3}_3)_x\right) -\kappa \left(\frac{\bar{\theta}_x}{\bar{v}} - \frac{(\widetilde{\theta}^{\bold{X}_1}_1)_x}{\widetilde{v}^{\bold{X}_1}_1}- \frac{\theta^D_x}{v^D} -  \frac{(\widetilde{\theta}^{\bold{X}_3}_3)_x}{\widetilde{v}^{\bold{X}_3}_3}\right)_x   \\
		& \ \ -\mu \left(\frac{\bar{u}^2_x}{\bar{v}} - \frac{\left((\widetilde{u}^{\bold{X}_1}_1)_x\right)^2} {\widetilde{v}^{\bold{X}_1}_1} - \frac{(u^D_x)^2}{v^D} - \frac{\left((\widetilde{u}^{\bold{X}_3}_3)_x\right)^2}{\widetilde{v}^{\bold{X}_3}_3}  \right),
	\end{aligned}
\end{align}
and the error terms associated with the viscous contact wave, denoted by $Q^C_1$ and $Q^C_2$, as defined in \eqref{est:QC}.

\subsection{Local existence of solutions}
Before proceeding, we mention that the existence of a local-in-time solution to \eqref{eq:NSF} (or to \eqref{eq:NSFS}) follows from the general result of Nash \cite{N62}.

\begin{proposition}\label{prop:local}
Let $\underline{v}$, $\underline{u}$, and $\underline{\theta}$ be smooth, monotone functions satisfying
	\[(\underline{v}(x), \underline{u}(x),\underline{\theta}(x))=(v_{\pm},u_{\pm},\theta_{\pm}),\quad\mbox{for}\quad \pm x\ge 1,\quad \underline{v}(0)>0 \quad \text{and} \quad \underline{\theta}(0)>0 .\]
Given constants $M_0, M_1, \underline{\kappa}_0, \overline{\kappa}_0, \underline{\kappa}_1,$ and $\overline{\kappa}_1$ satisfying  $$0<M_0<M_1, \quad  0<\underline{\kappa}_1<\underline{\kappa}_0<\overline{\kappa}_0<\overline{\kappa}_1,$$
there exists a constant $T_0 > 0$ such that the following statement holds.

If the initial data $(v_0, u_0, \theta_0)$ satisfy
	\begin{align*}
	&\|(v_0-\underline{v},u_0-\underline{u}, \theta_0-\underline{\theta})\|_{H^1(\R)}\le M_0,\\
	&0<\underline{\kappa}_0\le \rho_0(x), \theta_0(x)\le \overline{\kappa}_0,\quad \forall x\in\R,
\end{align*}  
then the problem \eqref{eq:NSF}–\eqref{eq:farfield} admits a unique solution $(v,u,\theta)$ on the time interval $[0, T_0]$ satisfying
	\begin{align*}
	\begin{aligned}
		&v -\underline{v}\in C([0,T_0];H^1(\R)), \\
		& (u-\underline{u}, \theta-\underline{\theta}) \in C([0,T_0];H^1(\R))\cap L^2(0,T_0;H^2(\R)),
	\end{aligned}
\end{align*} 
and the uniform bound
	\[\|(v-\underline{v},u-\underline{u}, \theta-\underline{\theta})\|_{L^\infty(0,T_0;H^1(\R))}\le M_1.\]
Moreover, the solution remains uniformly positive and bounded:
		\[\underline{\kappa}_1\le v(t,x), \theta(t,x)\le \overline{\kappa}_1, \qquad \forall (t,x) \in [0,T_0]\times \mathbb{R}.\]
\end{proposition}

\subsection{Construction of the weight function}
In order to localize the analysis around each shock wave, we define two distinct weight functions, $a_1$ and $a_3$, associated with the $1$- and $3$-shock profiles, respectively. For each $i = 1, 3$, we define the weight function as
\begin{align}\label{def_a}
	\begin{aligned}
		& a_1(x-\sigma_1t):=1 + \frac{1}{\sqrt{\delta_1}}(\tilde{v}_1(x-\sigma_1 t)-v_-), \\
		&a_3(x-\sigma_3t):=1 + \frac{1}{\sqrt{\delta_3}}(\tilde{v}_3(x-\sigma_3 t)- v^*).
	\end{aligned}
\end{align}	
where $\delta_1=|v_--v_*|$ and $\delta_3=|v_+-v^*|$ denote the strengths of $1$-shock and $3$-shock, respectively.

Note that for each $i=1, 3$, the weight function $a_i$ satisfies $\frac{1}{2}\le 1-\sqrt{\delta_i} \leq a_i\le 1+\sqrt{\delta_i}\le \frac{3}{2}$ and
\begin{align}\label{eq:ax}
	(a_1)_x = \frac{1}{\sqrt{\delta_1}}(\tilde{v}_1)_x, \qquad (a_3)_x = \frac{1}{\sqrt{\delta_3}}(\tilde{v}_3)_x,
\end{align}
from which we have 
\begin{align} \label{est:ax}
	|(a_i)_x| \sim \frac{1}{\sqrt{\delta_i}}|(\tilde{v}_i)_x|, \qquad \text{and so }\qquad 	\|(a_i)_x\|_{L^{\infty}(\mathbb{R})}  \leq \delta_i\sqrt{\delta_i}, \quad \|(a_i)_x\|_{L^{1}(\mathbb{R})}=\sqrt{\delta_i}.
\end{align}

In addition, 
we have
\[\sigma_ia_i'(x-\sigma_it)>0, \quad \ \text{for each} \ i=1, 3.\]

To capture the influence of both shifted shock components within the composite wave, we define the following combined weight function obtained from the corresponding shifted weights:
\begin{align}\label{def:a}
	a^{\bold{X}_1, \bold{X}_3}(t,x): = a_1^{\bold{X}_1}(x-\sigma_1t) +  a_3^{\bold{X}_3}(x-\sigma_3t) -1.
\end{align}
Here and in the sequel, for an arbitrary function $g:\mathbb{R}\to\mathbb{R}$, we make use of the abbreviated notation defined below:
\[g^{\bX_i}(\cdot):=g(\cdot-\bX_i(t)), \quad i=1,3.\] 
Note that $\frac{1}{2}\le 1-2\sqrt{\delta_0}\le a^{\bold{X}_1, \bold{X}_3}\le 1+2\sqrt{\delta_0}\le\frac{3}{2}$.

\subsection{Construction of shifts}
To capture the stability dynamics, the viscous shock profiles must be adjusted by suitable shifts.
We therefore construct the shift functions explicitly by defining $(\bold{X}_1, \bold{X}_3)$ as the solutions of the following system of ODEs:

\begin{equation}\label{def:shift}
	\begin{cases}
		&\dot{\bold{X}}_1= -\frac{M_1}{\delta_1} \di\int_{\mathbb{R}}a^{\bold{X}_1, \bold{X}_3} \left[(\widetilde{u}_1^{\bold{X}_1})_x(u-\bar{u}) +\frac{(\widetilde{v}_1^{\bold{X}_1})_x\bar{p}}{\bar{v}}(v-\bar{v}) +\frac{R}{\gamma-1}\frac{(\widetilde{\theta}_1^{\bold{X}_1})_x}{\bth}(\theta-\bth)  \right]dx,  \\
		& \dot{\bold{X}}_3 = -\frac{M_3}{\delta_3} \di\int_{\mathbb{R}}a^{\bold{X}_1, \bold{X}_3} \left[(\widetilde{u}_3^{\bold{X}_3})_x(u-\bar{u}) +\frac{(\widetilde{v}_3^{\bold{X}_3})_x\bar{p}}{\bar{v}}(v-\bar{v}) +\frac{R}{\gamma-1}\frac{(\widetilde{\theta}_3^{\bold{X}_3})_x}{\bth}(\theta-\bth)  \right]dx,  \\
		&	\bold{X}_1(0) =\bold{X}_3(0) =0.
	\end{cases}
\end{equation}
Here, $a^{\bold{X}_1, \bold{X}_3}$ denotes the shifted weight function introduced in \eqref{def:a}, and the constants
$$M_1 := \frac{3}{2}\frac{\alpha_1}{(\sigma_-)^2} \left(1+ \frac{2\kappa(\gamma-1)^2}{\mu R \gamma}\right) \quad \text{and} \quad M_3:= \frac{3}{2}\frac{\alpha_3}{(\sigma^*)^2}\left(1+ \frac{2\kappa(\gamma-1)^2}{\mu R \gamma}\right),$$
are specifically chosen for use in the proof of Proposition \ref{prop:main}. Following the argument of \cite[Lemma 3.2]{KVW-NSF}, we can ensure the existence of a Lipschitz continuous solution to the system of ODEs \eqref{def:shift}, provided that $v$ and $\theta$ remain uniformly positive and bounded, and that $u$ is bounded as well.
These conditions are ensured by Proposition \ref{prop:local} together with the a priori assumption \eqref{apriori_small}.

\subsection{A priori estimate}
We next formulate the central proposition that provides the a priori estimates required to derive the large-time asymptotics of the system.
\begin{proposition}\label{prop:main}
Let $(v_+,u_+,\theta_+)\in\mathbb{R}_+\times\mathbb{R}\times\mathbb{R}_+$ be a given constant state.
	Then there exist positive constants $C_0$, $\delta$, and $\varepsilon$ such that the following statement holds.
	
Assume that $(v,u,\theta)$ is a solution of \eqref{eq:NSFS} on the interval $[0,T]$ for some $T>0$, and let $(\bar{v},\bar{u},\bar{\theta})$ denote the composite wave consisting of two shifted viscous shocks and a viscous contact discontinuity, as defined in \eqref{eq:CW}, where the shifts $(\bold{X}_1,\bold{X}_3)$ satisfy the system \eqref{def:shift}.
Suppose further that the amplitudes of the three wave components satisfy $$\delta_1, \delta_C,\delta_3 <\delta_0,$$
and that the perturbation satisfies the regularity properties
	\begin{align*}
	v-\bar{v} &\in C([0,T];H^1(\R)),\\
	u-\bar{u},\theta- \bar{\theta}&\in C([0,T];H^1(\R))\cap L^2(0,T;H^2(\R)),
\end{align*}
together with the smallness condition
	\begin{equation}\label{apriori_small}
	\|(v-\bar{v}, u-\bar{u}, \theta-\bar{\theta})\|_{L^\infty\left(0,T;H^1(\R)\right)}\le \e.
\end{equation}
Then the following estimates hold:
\begin{align*}
	\begin{aligned}
		&\sup_{t\in[0,T]}
		\|(v-\bar{v}, u-\bar{u},\theta-\bar{\theta})(t,\cdot)\|_{H^1(\R)} \\
		&\quad
		+ \left(\int_0^T \Big(
		\sum_{i\in\{1,3\}}\delta_i|\dot{\bold{X}}_i|^2
		+ \sum_{i\in\{1,3\}}\mathcal{G}_i^{S}
		+ \bold{D}_{v_1}
		+ \bold{D}_{u_1}
		+ \bold{D}_{\theta_1}
		+ \bold{D}_{u_2}
		+ \bold{D}_{\theta_2}
		\Big)\,ds \right)^{\frac12} \\
		&\le C_0
		\|(v_0-\bar{v}(0,\cdot),u_0-\bar{u}(0,\cdot),\theta_0-\bar{\theta}(0,\cdot))\|_{H^1(\R)}
		+ C_0\,\delta_0^{\frac12}.
	\end{aligned}
\end{align*}

In particular, for all $t\in[0,T]$, the shift functions satisfy
\[
|\dot{\bold{X}}_1(t)| + |\dot{\bold{X}}_3(t)|
\le C_0\,\|(v-\bar{v},u-\bar{u}, \theta-\bar{\theta})(t,\cdot)\|_{L^\infty(\R)}.
\]

Here, the constant $C_0$ is independent of $T$. 

Furthermore, the dissipation and localized good terms are defined by
	
\begin{align}
\begin{aligned}\label{good_terms}
		&\mathcal{G}_i^{S}: = \int_{\mathbb{R}} \left|(\widetilde{v}_i)^{\bold{X}_i}_x\right| \left|\phi_i (v-\bar{v}, u-\bar{u},\theta-\bth )\right|^2 dx,  \\
		&\bold{D}_{v_1}:=\int_{\R}|(v-\bar{v})_{x}|^2\,dx, \quad  \bold{D}_{u_1}:=\int_{\R}|(u-\bar{u})_{x}|^2\,dx, \quad  \bold{D}_{u_2}:=\int_{\R}|(u-\bar{u})_{xx}|^2\,dx, ,  \\ 
		& \bold{D}_{{\theta}_1} := \int_{\R}|(\theta-\bar{\theta})_{x}|^2\,dx,  \quad \bold{D}_{\theta_2}:=\int_{\R}|(\theta-\bar{\theta})_{xx}|^2\,dx,
		\end{aligned}
	\end{align}
where $\phi_1$ and $\phi_3$ are the cutoff (localization) functions given by
	\begin{equation}\label{def:cutoff}
		\phi_1(t,x)	:= \left\{\begin{array}{ll}
			1 & \text{if} \ x<\frac{\bold{X}_1(t)+ \sigma_1 t}{2},\\
			0 & \text{if} \ x>\frac{\bold{X}_3(t)+ \sigma_3 t}{2}, \\
			\text{linearly decreasing from 1 to 0} &  \text{if} \ \frac{\bold{X}_1(t)+ \sigma_1 t}{2}\leq x\leq \frac{\bold{X}_3(t)+ \sigma_3 t}{2},
		\end{array}\right.
		\phi_3(t,x) := 1-\phi_1(t,x).
	\end{equation}
\end{proposition}

The proof of this key proposition will be given in Sections \ref{relative estimate} and \ref{higher order}.
In what follows, we show how Proposition \ref{prop:main} leads directly to the proof of Theorem \ref{main theorem}.

\subsection{Conclusion}
Building upon Propositions \ref{prop:local} and \ref{prop:main}, the desired large-time behavior \eqref{est:longtime} follows from a standard continuation argument.
The proof proceeds along classical lines and closely parallels the approach developed in \cite{KVW23}.
For this reason, we omit the detailed steps and move directly to completing the proof of Theorem \ref{main theorem}.
Consequently, the remainder of this paper is devoted to establishing Proposition \ref{prop:main}, which provides the essential a priori estimates underlying the argument.

\section{Zeroth Order Estimates}\label{relative estimate}
The objective of this section is to derive $L^2$–type estimates for the perturbation by employing the $a$–contraction method with shifts.
This estimate constitutes a crucial step in the proof of Proposition \ref{prop:main}.
More precisely, the main result of this section is summarized in the lemma stated below, to which the subsequent analysis is devoted.
\begin{lemma}\label{lm_zeroorder}
	Assume that the hypotheses of Proposition \ref{prop:main} are satisfied.
	Then there exists a constant $C>0$ such that, for all $t\in[0,T]$, the following estimate holds.
\begin{align}\label{est_zeroorder}
	\begin{aligned}
		&  \int_{\mathbb{R}}a\bar{\theta}\eta\left(U|\bar{U}\right)dx + \int_{0}^{t}  \left(\sum_{i\in\{1,3\}} \delta_i|\dot{\bold{X}_i}|^{2}    + \sum_{i\in\{1,3\}}\mathcal{G}_i^{S} +\bold{D}_{u_1} + \bold{D}_{\theta_1} \right) ds \\
		 	\leq & C\int_{\mathbb{R}}a\bar{\theta}\eta\left(U|\bar{U}\right)dx\bigg|_{t=0} 
		+ C\delta_1^{\frac{4}{3}}\left(\delta_3^{\frac{4}{3}} e^{-C\delta_3 t}  + \delta_C^{\frac{4}{3}}e^{-Ct}\right) +  C\delta_3^{\frac{4}{3}}\left(\delta_1^{\frac{4}{3}} e^{-C\delta_1 t}  + \delta_C^{\frac{4}{3}}e^{-Ct}\right)   \\
		& + C\delta_0 \int_{0}^{t}\left\|(v-\bar{v})_x\right\|_{L^2(\mathbb{R})}^2ds+ C\delta_0,   
	\end{aligned}
\end{align}	
	where $\mathcal{G}_i^{S}$, $\bold{D}_{u_1}$ and $\bold{D}_{\theta_1}$ are the terms defined in \eqref{good_terms}.
\end{lemma}

\subsection{Wave interaction estimates}
Our first step is to establish basic bounds for the interaction terms $Q^I_i$ appearing in \eqref{def:Q}.

From the a priori bound \eqref{apriori_small} together with the Sobolev embedding, it follows that
$$	\|(v-\bar{v}, u-\bar{u}, \theta-\bar{\theta})\|_{L^\infty\left((0,T)\times \R\right)}\le \e, \quad \text{and so}\quad v,u,\theta \in L^\infty\left((0,T)\times \R\right).$$
By applying Lemma \ref{lem:vs} to the ODE system \eqref{def:shift}, we deduce that
\begin{align}
	\begin{aligned}\label{est:X}
		|\dot{\bold{X}}_1(t)| + |\dot{\bold{X}}_3(t)| &\leq \sum_{i\in\{1,3\}}\frac{C}{\delta_i}\|\left(v-\bar{v}, u-\bar{u}, \theta-\bar{\theta}\right)\|_{L^\infty\left(\R\right)} \int_{\R} |(\widetilde{v}_i)_x| dx  \\
		& \leq \|\left(v-\bar{v}, u-\bar{u}, \theta-\bar{\theta}\right)\|_{L^\infty\left(\R\right)}<\varepsilon.
	\end{aligned}
\end{align}
Therefore, from \eqref{est:X} and the smallness assumption on $\varepsilon$, it follows that
\[\bold{X}_1 (t) \leq -\frac{\sigma_1}{2}t, \qquad \bold{X}_3(t) \geq -\frac{\sigma_3}{2}t, \quad t>0,\]
or equivalently, 
\[\bold{X}_1 (t)+ \sigma_1t  \leq \frac{\sigma_1}{2}t, \qquad \bold{X}_3(t) +\sigma_3t \geq \frac{\sigma_3}{2}t, \quad t>0,\]
which proves \eqref{est:waveseparation}.

\begin{lemma}\label{lm:interaction}
	Let $Q^I_i$ $(i=1,2)$ be as given in \eqref{def:Q}.
	Under the assumptions of Proposition \ref{prop:main}, there exists a constant $C>0$, independent of $T$, $\delta_1$, $\delta_3$, and $\delta_C$, for which the following bound is valid for all $t\le T$:
\begin{align*}\label{est:interaction}
	\begin{aligned}
			&\|Q^I_i\|_{L^2(\mathbb{R})} \leq  C \delta_1(\delta_C + \delta_3 )e^{-C\delta_1 t} + C\delta_C(\delta_1 + \delta_3)e^{-Ct} + C \delta_3(\delta_1 + \delta_C )e^{-C\delta_3 t}, \ \text{for} \ i=1, 2,\\
& \|(\widetilde{v}^{\bold{X}_1}_1)_x||(v^D-v_*,\theta^D-\theta_*)|\|_{L^2(\mathbb{R})} + \||(\widetilde{v}^{\bold{X}_1}_1)_x||(\widetilde{v}^{\bold{X}_3}_3-v^*,\widetilde{\theta}^{\bold{X}_3}_3-\theta^*)|\|_{L^2(\mathbb{R})}  \\
		  &\leq  C\delta_1^{\frac{3}{2}}\left(\delta_3 e^{-C\delta_3 t}  + \delta_Ce^{-Ct}\right), \\
		& \|(\widetilde{v}^{\bold{X}_3}_3)_x||(v^D-v^*,\theta^D-\theta^*)|\|_{L^2(\mathbb{R})} + \||(\widetilde{v}^{\bold{X}_3}_3)_x||(\widetilde{v}^{\bold{X}_1}_1-v_-*\widetilde{\theta}^{\bold{X}_1}_1-\theta_*)|\|_{L^2(\mathbb{R})}  \\
		  &\leq  C\delta_3^{\frac{3}{2}}\left(\delta_1 e^{-C\delta_1t}  + \delta_Ce^{-Ct}\right).
	\end{aligned}
\end{align*}
\end{lemma}

\begin{proof}
We omit the details, since the proof closely follows that of \cite[Lemma 4.2]{KVW-NSF}, with minor adjustments specific to our context. The interested reader may refer to that work for the detailed argument.
\end{proof}
\begin{lemma}\label{lem:cutoffinteraction}
Define $\phi_1$ and $\phi_3$ according to \eqref{def:cutoff}.
Then one can find positive constants $\delta_0$ and $C$ such that, for all small parameters $\delta_1, \delta_C, \delta_3 \in (0,\delta_0)$, the estimates below hold true.
	\begin{align*}
       \begin{aligned}
       	& \phi_3 |(\widetilde{v}^{\bold{X}_1}_1)_x| \leq C\delta^2_1 \exp(-C\delta_1 t), \quad \phi_1 |(\widetilde{v}^{\bold{X}_3}_3)_x| \leq C\delta^2_3 \exp(-C\delta_3 t), \quad t>0, \quad x\in\R,  \\
       	& \int_{\R}\phi_3 |(\widetilde{v}^{\bold{X}_1}_1)_x| dx \leq C\delta_1 \exp(-C\delta_1 t), \quad \int_{\R}\phi_1|(\widetilde{v}^{\bold{X}_3}_3)_x|  dx \leq C\delta_3 \exp(-C\delta_3 t), \quad t>0.
       \end{aligned}
	\end{align*}
\end{lemma}
\begin{proof}
The details are omitted, as the proof proceeds in the same manner as that of \cite[Lemma 3.3]{HKK23}.
\end{proof}

$\bullet$ Notations:
For simplicity of presentation, we shall employ the following notations in what follows.

\noindent
1. We will suppress the dependence on the shifts.
\[a(t,x):= a^{\bold{X}_1,\bold{X}_3}(t,x).\]
and for each $i=1, 3$,
\begin{align*}
&\widetilde{v}_i(t,x):=\widetilde{v}_i^{\bold{X}_i}(x-\sigma_it)=\widetilde{v}_i(x-\sigma_it -\bold{X}_i(t)), \\ 
&\widetilde{u}_i(t,x):=\widetilde{u}_i^{\bold{X}_i}(x-\sigma_it)=\widetilde{u}_i(x-\sigma_it -\bold{X}_i(t)),  \\
&\widetilde{\theta}_i(t,x):=\widetilde{\theta}_i^{\bold{X}_i}(x-\sigma_it)=\widetilde{\theta}_i(x-\sigma_it -\bold{X}_i(t)).
\end{align*}
2. In analyzing the evolution of the relative entropy, we adopt the notation $$U=(v, u, \theta)^t, \qquad \bar{U}=(\bar{v}, \bar{u}, \bar{\theta})^t,$$
where $U$ denotes the solution to the NSF system \eqref{eq:NSFS},
and $\bar{U}$ represents the composite wave consisting of a 1–viscous shock shifted by $\bold{X}_1$,
a 2–viscous contact wave, and a 3–viscous shock shifted by $\bold{X}_3$, as described below.

\begin{equation} \label{eq:CW}
	\bar{U}(t,x):=\begin{pmatrix}
		\bar{v}(t,x)\\
		\bar{u}(t,x)\\
		\bar{\theta}(t,x)
	\end{pmatrix}:=
	\begin{pmatrix}
		\widetilde{v}_1(t,x) + v^{D}(t,x) + \widetilde{v}_3(t,x)-v_* -v^*\\
		\widetilde{u}_1(t,x) + u^{D}(t,x) + \widetilde{u}_3(t,x)-u_*-u^*\\
		\widetilde{\theta}_1(t,x) + \theta^{D}(t,x) + \widetilde{\theta}_3(t,x)-\theta_* -\theta^*\\
	\end{pmatrix}.
\end{equation}

3. We adopt the following notations to make the computations presented after subsection \ref{leading order} more concise:
\begin{align*}
&v^{[1]}:=v_-, \ \theta^{[1]}:=\theta_-, \ p^{[1]}:=p_-,  \ and \  \sigma^{[1]}:=\sigma_-, \\
&v^{[3]}:=v^*, \ \theta^{[3]}:=\theta^*, \ p^{[3]}:=p^*,  \ and \  \sigma^{[3]}:=\sigma^*. \\
\end{align*}

\subsection{Relative entropy method}

Following \cite{KVW-NSF}, we introduce the relative entropy between $U$ and $\bar{U}$ by
\[
\eta(U \mid \bar{U})
= R\Phi\!\left(\frac{v}{\bar{v}}\right)
+ \frac{R}{\gamma-1}\Phi\!\left(\frac{\theta}{\bar{\theta}}\right)
+ \frac{(u-\bar{u})^2}{2\bar{\theta}},
\qquad
\Phi(z):=z-1-\ln z.
\]
Accordingly, the $\bar{\theta}$-weighted relative entropy is given by
\begin{equation}\label{eq:WRE}
	\bar{\theta}\,\eta(U \mid \bar{U})
	= R\bar{\theta}\Phi\!\left(\frac{v}{\bar{v}}\right)
	+ \frac{R\bar{\theta}}{\gamma-1}\Phi\!\left(\frac{\theta}{\bar{\theta}}\right)
	+ \frac{(u-\bar{u})^2}{2}.
\end{equation}

We next derive the evolution identity for the weighted relative entropy with weight $a(t,x)\bar{\theta}(t,x)$.
\begin{lemma} \label{lmREC}
Let $a$ be the weight function introduced in \eqref{def_a}, and let $U$ and $\bar{U}$ denote, respectively, the solution of \eqref{eq:NSFS} and the composite wave described in \eqref{eq:CW}.
Then we have
\begin{equation*} 
	\begin{aligned}
&{d \over dt}\int_{\mathbb{R}} a(t,x)\bar{\theta}\eta (U(t,x))|\bar{U}(t,x)) dx=\sum_{i\in\{1,3\}}\sX_i(t)\bold{Y}_i(U)+\mathcal{J}^{bad}(U)-\mathcal{J}^{good}(U),
\end{aligned}
\end{equation*}
where
\begin{equation*} \label{4.27} 
	\begin{split}
		\bold{Y}_i(U):=& -\int_{\mathbb{R}} (a_i)_x\bth \eta (U|\bar{U}) dx +\int_{\mathbb{R}} a\left[R(\widetilde{\theta}_i)_x\Phi\left({v\over \bar{v}}\right)+\frac{R}{\gamma-1}(\widetilde{\theta}_i)_x\Phi\left({\theta \over \bar{\theta}}\right) \right]dx  \\
		& + \int_{\mathbb{R}} a\left[\frac{(\widetilde{v}_i)_x\bar{p}}{\bar{v}}(v-\bar{v})+ \frac{R}{\gamma-1}\frac{(\widetilde{(\theta}_i)_x}{\bar{\theta}}(\theta -\bar{\theta})+(\widetilde{u}_i)_x(u-\bar{u})\right] dx, \\
		\mathcal{J}^{bad}(U):=&\int_{\mathbb{R}}a \left[R\left(-\sigma_1(\widetilde{\theta}_1)_x  +\theta^D_t-\sigma_3(\widetilde{\theta}_3)_x\right)\Phi\left({v\over \bar{v}}\right) - \frac{\bar{p}\bar{u}_x}{v\bar{v}}(v-\bar{v})^2 \right] dx \\
		& +\int_{\mathbb{R}} a\left[\frac{R}{\gamma-1}\left(-\sigma_1(\widetilde{\theta}_1)_x  +\theta^D_t-\sigma_3(\widetilde{\theta}_3)_x\right)\Phi\left({\theta \over \bar{\theta}}\right)  -\frac{\bar{u}_x}{\theta}(\theta-\bth)(p-\bar{p}) +\bar{p}\bar{u}_x\frac{(\theta-\bar{\theta})^2}{\theta\bar{\theta}}\right] dx    \\
		& +\int_{\mathbb{R}}  a\left[\mu \left(\frac{u_x}{v}-\frac{\bar{u}_x}{\bar{v}}\right)(u-\bar{u}) +\kappa\frac{\theta -\bar{\theta}}{\theta}\left(\frac{\theta_x}{v}-\frac{\bar{\theta}_x}{\bar{v}}\right)-(p-\bar{p})(u-\bar{u})\right]_xdx  \\
		& +\int_{\mathbb{R}} a\bigg[- \mu(u-\bar{u})_x\bar{u}_x\left(\frac{1}{v}-\frac{1}{\bar{v}}\right) + \kappa\frac{\theta-\bar{\theta}}{\theta^2}\theta_x\left(\frac{\theta_x}{v}-\frac{\bar{\theta}_x}{\bar{v}}\right)  - \kappa\frac{(\theta -\bar{\theta})_x}{\theta}\bar{\theta}_x\left(\frac{1}{v}-\frac{1}{\bar{v}}\right)  \\
		& \qquad \qquad +\mu\frac{\theta -\bar{\theta}}{\theta}\bar{\theta}_x\left(\frac{u^2_x}{v}-\frac{\bar{u}^2_x}{\bar{v}}\right) -(u-\bar{u})Q_1 -\left(\frac{\theta}{\bar{\theta}}-1\right)Q_2 \bigg]dx  \\ 
		& + \int_{\mathbb{R}} a \bigg[-\kappa\frac{(\theta-\bar{\theta})^2}{\theta \bar{\theta}}\left(\frac{\bar{\theta}_x}{\bar{v}}\right)_x -\mu\frac{(\theta-\bar{\theta})^2}{\theta \bar{\theta}}\frac{\bar{u}^2_x}{\bar{v}}\bigg] dx,  \\
		\mathcal{J}^{good}(U):=& \sum_{i\in\{1,3\}}\sigma_i\int_{\mathbb{R}}(a_i)_x\bar{\theta}\eta\left(U|\bar{U}\right)dx +\int_{\mathbb{R}}  a\left[\frac{\mu}{v} |(u-\bar{u})_x|^2 +\frac{\kappa}{v\theta}\left|(\theta-\bar{\theta})_x\right|^2 \right]dx.
		\end{split}
\end{equation*}
\end{lemma}

\begin{remark} Since $\sigma_ia'_i>0$,  $-\mathcal{J}^{good}$ consists of good terms,  while $\mathcal{J}^{bad}$ consists of bad terms.
\end{remark}
\begin{proof}
Since the proof is almost the same as \cite[Lemma 4.3]{KVW-NSF}, we refer to this literature for the detailed proof. 
\end{proof}

\subsection{Decompositions}
First, we separate the first and second terms of $\mathcal{J}^{bad}$ into the principal term $\bold{B}_1$ and the remaining higher-order parts.
\begin{equation}\label{defb1}
	\begin{aligned}
	 \bold{B}_1(U):=  &\sum\limits_{i\in \{1, 3\}}\int_{\mathbb{R}} a(\widetilde{v}_i)_x\left[ \frac{(\gamma+1)p^{[i]}\sigma^{[i]}}{2(v^{[i]})^2} |v-\bar{v}|^2 + \frac{\sigma^{[i]}R}{2v^{[i]}\theta^{[i]}}(\theta-\bth)^2-\frac{\sigma^{[i]}p^{[i]}}{2v^{[i]}\theta^{[i]}}(v-\bar{v})(\theta-\bth)\right] dx.\\
\end{aligned}
\end{equation}
We begin by analyzing the first term in $\mathcal{J}^{bad}$:
\[\int_{\mathbb{R}} a   \underbrace{\left[R(-\sigma_1\left((\widetilde{\theta}_1)_x  +\theta^D_t-\sigma_3(\widetilde{\theta}_3)_x\right)\Phi\left({v\over \bar{v}}\right) - \frac{\bar{p}\bar{u}_x}{v\bar{v}}(v-\bar{v})^2 \right] }_{=:J_2}dx. \]
Noting that $\bar{u}_x = (\widetilde{u}_1)_x + u_x^C + (\widetilde{u}_3)_x$, we obtain
\begin{align*}
 \begin{aligned}
{J}_2   =&  \underbrace{ -R\sigma_1(\widetilde{\theta}_1)_x\Phi\left({v\over \bar{v}}\right) - \frac{\bar{p}(\widetilde{u}_1)_x}{v\bar{v}}(v-\bar{v})^2}_{=:J_{21}} \\
&+ \underbrace{R\theta^D_t\Phi\left({v\over \bar{v}}\right)- \frac{\bar{p}u_x^C}{v\bar{v}}(v-\bar{v})^2}_{=:J_{22}}  
\underbrace{-R\sigma_3(\widetilde{\theta}_3)_x\Phi\left({v\over \bar{v}}\right) - \frac{\bar{p}(\widetilde{u}_3)_x}{v\bar{v}}(v-\bar{v})^2}_{=:J_{23}}      
 \end{aligned}
\end{align*}
Since $\Phi(1)=\Phi'(1)=0$ and $\Phi''(1)=1$, a Taylor expansion of $\Phi(z)$ at $z=1$ yields
\begin{align}\label{eq:Phiexpan}
	\Phi\left({v\over \bar{v}}\right) = \frac{(v-\bar{v})^2}{2\bar{v}^2} + O(|v-\bar{v}|^3).
\end{align}

Using the above relation together with $\eqref{eq:VCD}_3$, it follows that
\[J_{22} \leq C\left(|u_x^C|  + |\theta_{xx}^C| + |\theta_x^C||v_x^C| + |Q_2^C|\right)|v-\bar{v}|^2,\]

which, combined with \eqref{est:vcdproperty} and \eqref{est:QC}, implies
\[J_{22} \leq C\delta_C(1+t)^{-1}e^{-\frac{C_1|x|^2}{1+t}}(v-\bar{v})^2.\]

By applying \eqref{est:vu}, \eqref{est:vtheta}, \eqref{est:sigma}, and 
\begin{equation}\label{est:pp}
	\begin{split}
		|\bar{p}-p^{[1]}| & \leq C\left(|\bar{v}-v^{[1]}|+ |\bth-\theta^{[1]}|\right)  \\
		& \leq C\left(|\widetilde{v}_1- v_-| + |v^D-v_*| + |\widetilde{v}_3-v^*| + |\widetilde{\theta}_1- v_-| + |\theta^D-v_*| + |\widetilde{\theta}_3-v^*| \right),  \\
		& \leq C(\delta_1 + \delta_C + \delta_3) \leq C\delta_0,
	\end{split}
\end{equation}
combining this with \eqref{eq:Phiexpan}, we obtain
\[J_{21} \leq \frac{(\gamma+1)p^{[1]}\sigma^{[1]}}{2(v^{[1]})^2}(\widetilde{v}_1)_x |v-\bar{v}|^2 + C(\delta_0 + \varepsilon)(\widetilde{v}_1)_x |v-\bar{v}|^2.\]

Likewise,
\[J_{23} \leq \frac{(\gamma+1)p^{[3]}\sigma^{[3]}}{2(v^{[3]})^2}(\widetilde{v}_3)_x |v-\bar{v}|^2 + C(\delta_0 + \varepsilon)(\widetilde{v}_3)_x |v-\bar{v}|^2.\]

Analogously, we treat the second contribution of $\mathcal{J}^{bad}$:
\[\int_{\mathbb{R}} a\underbrace{\left[\frac{R}{\gamma-1}\left(-\sigma_1(\widetilde{\theta}_1)_x  +\theta^D_t-\sigma_3(\widetilde{\theta}_3)_x\right)\Phi\left({\theta \over \bar{\theta}}\right)  -\frac{\bar{u}_x}{\theta}(\theta-\bth)(p-\bar{p}) +\bar{p}\bar{u}_x\frac{(\theta-\bar{\theta})^2}{\theta\bar{\theta}}\right]}_{=:J_3} dx \]

Using the identity 
\begin{equation}\label{eq:pp}
p-\bar{p} = \frac{R}{v}(\theta-\bth) - \frac{\bar{p}}{v}(v-\bar{v}),
\end{equation}
it follows that
\begin{align*}
	J_3 =&  \sum\limits_{i\in \{1, 3\}}\left[-\frac{R\sigma_i}{\gamma-1}(\widetilde{\theta}_i)_x\Phi\left({\theta \over \bar{\theta}}\right) -(\widetilde{u}_i)_x\frac{\theta-\bth}{\theta}\left(\frac{R}{v}(\theta-\bth) - \frac{\bar{p}}{v}(v-\bar{v})\right)+ \bar{p}(\widetilde{u}_i)_x\frac{(\theta-\bar{\theta})^2}{\theta\bar{\theta}}\right]  \\
	&  + \frac{R}{\gamma-1}(\theta^D)_t\Phi\left({\theta \over \bar{\theta}}\right) -(u^D)_x\frac{\theta-\bth}{\theta}\left(\frac{R}{v}(\theta-\bth) - \frac{\bar{p}}{v}(v-\bar{v})\right)+ \bar{p}(u^D)_x\frac{(\theta-\bar{\theta})^2}{\theta\bar{\theta}} \\
	\leq & \ \sum\limits_{i\in \{1, 3\}}\frac{\sigma^{[i]}(\widetilde{v}_i)_x}{2v^{[i]}\theta^{[i]}}\left(R(\theta-\bth)-2p^{[i]}(v-\bar{v})\right)(\theta-\bth) \\
	& +C\delta_C(1+t)^{-1}e^{-\frac{C_1|x|^2}{1+t}}|(v-\bar{v},\theta-\bth)|^2 +C(\delta_0 + \varepsilon)(|(\widetilde{v}_1)_x| + |(\widetilde{v}_3)_x|)|(v-\bar{v},\theta-\bth)|^2 .
	\end{align*}
	
Therefore, by combining the above estimates with the relation $\bar{p}= \frac{R\bth}{\bar{v}}$, we obtain
	\begin{align*}
		& \int_{\mathbb{R}} a(J_2 + J_3) dx   \\
		\leq & \ \sum\limits_{i\in \{1, 3\}}\int_{\mathbb{R}} a(\widetilde{v}_i)_x\left[ \frac{(\gamma+1)p^{[i]}\sigma^{[i]}}{2(v^{[i]})^2} |v-\bar{v}|^2 + \frac{\sigma^{[i]}R}{2v^{[i]}\theta^{[i]}}(\theta-\bth)^2-\frac{\sigma^{[i]}p^{[i]}}{2v^{[i]}\theta^{[i]}}(v-\bar{v})(p-\bth)\right] dx  \\
		& \ + C\delta_C(1+t)^{-1} \int_{\mathbb{R}} ae^{-\frac{C_1|x|^2}{1+t}}|(v-\bar{v},\theta-\bth)|^2 dx \\ 
        & \ +  C(\delta_0 + \varepsilon)\int_{\mathbb{R}}(|(\widetilde{v}_1)_x| + |(\widetilde{v}_3)_x|)|(v-\bar{v},\theta-\bth)|^2  dx.
	\end{align*}

Therefore, applying Lemma \ref{lmREC}, we deduce that
	\begin{equation}\label{est_RE}
	\begin{aligned}
		{d \over dt}& \int_{\mathbb{R}}a\bar{\theta}\eta\left(U|\bar{U}\right)dx \\
		&\leq  \sum_{i\in\{1,3\}}\bold{X}_i(t)\bold{Y}_i(U) + \sum_{i=1}^{5}\bold{B}_i + \bold{S}_1 + \bold{S}_2 -\bold{G}(U) -\bold{D}(U),
	\end{aligned}
	\end{equation}
	where $\bold{B}_1$ is as in \eqref{defb1}, and
	\begin{align*}
		\begin{aligned}
			& \bold{B}_2 :=\int_{\mathbb{R}} a_{x}(u-\bar{u})(p-\bar{p})dx,   \\
			&  \bold{B}_3 := -\int_{\mathbb{R}}  a_x\left[\mu \left(\frac{u_x}{v}-\frac{\bar{u}_x}{\bar{v}}\right)(u-\bar{u}) +\kappa\frac{\theta -\bar{\theta}}{\theta}\left(\frac{\theta_x}{v}-\frac{\bar{\theta}_x}{\bar{v}}\right)\right]dx,  \\
			& \bold{B}_4 :=\int_{\mathbb{R}} a\bigg[- \mu(u-\bar{u})_x\bar{u}_x\left(\frac{1}{v}-\frac{1}{\bar{v}}\right) + \kappa\frac{\theta-\bar{\theta}}{\theta^2}\theta_x\left(\frac{\theta_x}{v}-\frac{\bar{\theta}_x}{\bar{v}}\right) \\
			& \qquad\qquad \qquad -\kappa\frac{(\theta -\bar{\theta})_x}{\theta}\bar{\theta}_x\left(\frac{1}{v}-\frac{1}{\bar{v}}\right)  +\mu\frac{\theta -\bar{\theta}}{\theta}\bar{\theta}_x\left(\frac{u^2_x}{v}-\frac{\bar{u}^2_x}{\bar{v}}\right)\bigg]\,dx,  \\ 
            &\bold{B}_5 :=-\int_{\mathbb{R}} a \bigg[\kappa\frac{(\theta-\bar{\theta})^2}{\theta \bar{\theta}}\left(\frac{\bar{\theta}_x}{\bar{v}}\right)_x +\mu\frac{(\theta-\bar{\theta})^2}{\theta \bar{\theta}}\frac{\bar{u}^2_x}{\bar{v}}\bigg] dx, \\
			 & \bold{B}_6 :=C\delta_C(1+t)^{-1} \int_{\mathbb{R}} ae^{-\frac{C_1|x|^2}{1+t}}|(v-\bar{v},\theta-\bth)|^2 dx   \\
			 &\qquad\quad+  C(\delta_0 + \varepsilon)\int_{\mathbb{R}}(|(\widetilde{v}_1)_x| + |(\widetilde{v}_3)_x|)|(v-\bar{v},\theta-\bth)|^2  dx, \\
			 & \bold{S}_1:= -\int_{\mathbb{R}} a (u-\bar{u})Q_1 dx, \qquad\qquad  \bold{S}_2:= -\int_{\mathbb{R}} a\left(\frac{\theta}{\bar{\theta}}-1\right)Q_2 dx, \\
			 &\bold{G}(U):= \sum_{i\in\{1,3\}}\sigma_i\int_{\mathbb{R}}(a_i)_x\bar{\theta}\eta\left(U|\bar{U}\right)dx,   \\
			 & \bold{D}(U):=\int_{\mathbb{R}}  a\left[\frac{\mu}{v} |(u-\bar{u})_x|^2 +\frac{\kappa}{v\theta}\left|(\theta-\bar{\theta})_x\right|^2 \right]dx.
		\end{aligned}
	\end{align*}
For each $i=1,3$, we decompose the functional $\bold{Y}_i$ in the form,
	\begin{equation*}
	 \bold{Y}_i :=\sum_{j=1}^{6}\bold{Y}_{ij},   
	\end{equation*}
	where
	\begin{align*}
		\begin{split}
		&	\bold{Y}_{i1} := \int_{\mathbb{R}} a(\widetilde{u}_i)_x(u-\bar{u})dx,  \qquad 	\quad \quad \bold{Y}_{i2} := \int_{\mathbb{R}} a\frac{(\widetilde{v}_i)_x\bar{p}}{\bar{v}}(v-\bar{v})dx, \\
			&	\bold{Y}_{i3} := \int_{\mathbb{R}} a\frac{R}{\gamma-1}\frac{(\widetilde{\theta}_i)_x}{\bth}(\theta-\bth) dx,  \qquad 	\bold{Y}_{i4} := \int_{\mathbb{R}} a R(\widetilde{\theta}_i)_x\Phi\left(\frac{v}{\bar{v}}\right) dx,  \\
			&	\bold{Y}_{i5} := \int_{\mathbb{R}} a\frac{R}{\gamma-1}(\widetilde{\theta}_i)_x\Phi\left(\frac{\theta}{\bth}\right) dx, \qquad \bold{Y}_{i6} := -\int_{\mathbb{R}} (a_i)_x\bth \eta\left(U(t,x)|\bar{U}(t,x)\right)dx .
		\end{split}
	\end{align*}
It follows from the definition \eqref{def:shift} that
	\begin{align}\label{eq:decomposeX}
	\dot{\bold{X}}_i(t) = -\frac{M_i}{\delta_i}(\bold{Y}_{i1} +\bold{Y}_{i2}+ \bold{Y}_{i3} ),
	\end{align}
	which yields
	\begin{align}\label{eq_XY}
		\dot{\bold{X}}_i(t) \bold{Y}_i = -\frac{\delta_i}{M_i}|\dot{\bold{X}}_i|^{2} + \dot{\bold{X}}_i \sum_{j=4}^{6}\bold{Y}_{ij}.
	\end{align}
\subsection{Dominant term estimates}\label{leading order}
In this subsection, we deal with important terms in a sharp way. The main goal is to prove the below Lemma.
    
	\begin{lemma}\label{lma_part1}
	    There exists $C>0$ such that 
	    \begin{align*}
	    	\begin{aligned}
	   & -\sum\limits_{i\in \left\{1, 3\right\}}\frac{\delta_i}{2M_i}|\dot{\bold{X}_i}|^{2}+ \bold{B}_1 + \bold{B}_2 -\bold{G} -\frac{3}{4}\bold{D}  \\
	    \leq & C\sum_{i\in\{1,3\}}\mathcal{G}_i^{S} + \delta_1^{\frac{4}{3}}\left(\delta_3^{\frac{4}{3}} e^{-C\delta_3 t}  + \delta_C^{\frac{4}{3}}e^{-Ct}\right) +  \delta_3^{\frac{4}{3}}\left(\delta_1^{\frac{4}{3}} e^{-C\delta_1 t}  + \delta_C^{\frac{4}{3}}e^{-Ct}\right) \\
	   & +  \left(\sum_{i\in\{1,3\}}C\delta^2_i\exp(-C\delta_it) + \frac{C}{\delta_*t^2}\right)\int_{\mathbb{R}}\eta\left(U|\bar{U}\right)dx,
	   \end{aligned}
	    \end{align*}
	  where 
	  \[\mathcal{G}_i^{S}= \int_{\mathbb{R}} \left|(\widetilde{v}_i)_x\right| \left|\phi_i (v-\bar{v}, u-\bar{u},\theta-\bth )\right|^2 dx.\]
	\end{lemma}
		First, for any fixed $t\in \left[0, T\right]$, we define new variables $y_1, y_3$ and $w_1, w_3$ as below:
		\begin{align*}
			\begin{aligned}
				y_1 := -{\tilde{v}_1(x-\sigma_1 t-\bX_1(t))-v^{[1]}\over \delta_1} \qquad \text{and} \qquad y_3 := {\tilde{v}_3(x-\sigma_3 t-\bX_3(t))-v^{[3]}\over \delta_3}.
			\end{aligned}
		\end{align*}
		In fact, for $i=1, 3$, $y_i: \mathbb{R} \rightarrow (0,1)$ is strictly increasing in $\xi_i = x-\sigma_i t- \bX_i(t)$ since
		\begin{align}\label{eq:dy}
			{dy_1\over d\xi_1} = -\frac{(\widetilde{v}_1)_{x}}{\delta_1} > 0, \qquad {dy_3\over d\xi_3} = \frac{(\widetilde{v}_3)_{x}}{\delta_3} > 0.
		\end{align}
		Besides,
		\[\lim_{\xi_i \rightarrow -\infty} y_i = 0, \qquad  \lim_{\xi_i \rightarrow +\infty} y_i = 1. \]
		In the new variables, we apply Lemma \ref{Poincare} to each perturbation $w_i$ localized by $\phi_i$:
		\begin{align*}
			& w_1 := \phi_1(t, x)\bigg[u(t,x)- \bigg(\tilde{u}_1(x-\sigma_1 t-\bX_1(t)) + u^D(t,x)+ \tilde{u}_3\left(x-\sigma_3t- \bX_3(t)\right) -u_* -u^*\bigg)\bigg],   \\
			& w_3 := \phi_3(t, x)\bigg[u(t,x)- \bigg
			(\tilde{u}_1(x-\sigma_1 t- \bX_1(t) ) + u^D(t,x)+\tilde{u}_3(x-\sigma_3t-\bX_3(t)) -u_* -u^*\bigg)\bigg].  \\
		\end{align*}
	
		\noindent$\bullet$ (Estimate on  $\bold{B}_2-\bold{G}$):

		At first, by \eqref{eq:WRE} and \eqref{eq:pp}, we have 
		\begin{align*}
			\begin{aligned}
		& (u-\bar{u})(p-\bar{p})-\sigma_1 \bar{\theta}\eta(U | \bar{U}) \\
		 =&(u-\bar{u})\left(\frac{R}{v}(\theta-\bth) - \frac{\bar{p}}{v}(v-\bar{v})\right) -\sigma_1\left[R\bar{\theta}\Phi\left({v\over \bar{v}}\right) + \frac{R\bar{\theta}}{\gamma-1}\Phi\left({\theta \over \bar{\theta}}\right) + {(u-\bar{u})^{2}\over 2}\right] .
\end{aligned}
\end{align*}
By virtue of \eqref{est:sigma}, \eqref{eq:Phiexpan}, and \eqref{est:pp}, we have 
	\begin{align*}
	\begin{aligned}
& (u-\bar{u})(p-\bar{p})-\sigma_1 \bar{\theta}\eta(U | \bar{U}) \\
 \leq & (u-\bar{u})\left[\frac{R}{v^{[1]}}(\theta-\bth) - \frac{p^{[1]}}{v^{[1]}}(v- \bar{v})\right] + \frac{R\theta^{[1]}\sigma^{[1]}}{2(v^{[1]})^2} |v-\bar{v}|^2 + \frac{R\sigma^{[1]}}{2(\gamma-1)\theta^{[1]}}(\theta-\bth)^2  \\
& + \frac{\sigma^{[1]}}{2}(u-\bar{u})^2 + C\left(|v-\bar{v}| + |\bar{v}-v^{[1]}| + |\bth-\theta^{[1]}|\right)\cdot|(v-\bar{v}, u-\bar{u}, \theta-\bth)|^2  \\
=& \frac{R\theta^{[1]}\sigma^{[1]}}{2(v^{[1]})^2} \left[(v-\bar{v}) - \frac{u-\bar{u}}{\sigma^{[1]}}\right]^2 + \frac{R\sigma^{[1]}}{2(\gamma-1)\theta^{[1]}}\left[(\theta-\bth)+ \frac{(\gamma-1)\theta^{[1]}}{\sigma^{[1]}v^{[1]}}(u-\bar{u})\right]^2    \\
& + C\left(|v-\bar{v}| +\delta_1 +  |(\widetilde{v}_3-v^{[3]},\widetilde{\theta}_3-\theta^{[3]})| + |(v^D-v_*, \theta^D-\theta_*)|\right)\cdot|(v-\bar{v},u-\bar{u},\theta-\bth)|^2  ,\\
\end{aligned}
\end{align*}
where the last equality is obtained by applying
		\[ \frac{R\theta^{[1]}\sigma^{[1]}}{2(v^{[1]})^2}  +\frac{R\sigma^{[1]}}{2(\gamma-1)\theta^{[1]}}= \frac{R\gamma\theta^{[1]}}{2\sigma^{[1]}(v^{[1]})^2}  =\frac{\sigma^{[1]}}{2} \qquad \text{with}\qquad  \sigma^{[1]} = \frac{\sqrt{\gamma R\theta^{[1]}}}{v^{[1]}}.\] 
Similarly, we have
		\begin{equation*}
			\begin{aligned}
				& (u-\bar{u})(p-\bar{p})-\sigma_3 \bar{\theta}\eta(U | \bar{U}) \\
				\leq & \frac{R\theta^{[3]}\sigma^{[3]}}{2(v^{[3]})^2} \left[(v-\bar{v}) - \frac{u-\bar{u}}{\sigma^{[3]}}\right]^2 + \frac{R\sigma^{[3]}}{2(\gamma-1)\theta^{[3]}}\left[(\theta-\bth)+ \frac{(\gamma-1)\theta^{[3]}}{\sigma^{[3]}v^{[3]}}(u-\bar{u})\right]^2    \\
				& + C\left(|v-\bar{v}| +\delta_3 +  |(\widetilde{v}_1-v^{[1]},\widetilde{\theta}_1-\theta^{[1]})| + |(v^D-v^*, \theta^D-\theta^*)|\right)\cdot|(v-\bar{v},u-\bar{u},\theta-\bth)|^2.\\
			\end{aligned}
		\end{equation*}
As a consequence, we have
		\begin{align}
        \begin{aligned}\label{est:BG}
			\bold{B}_2(U) - \bG(U) \leq &-\sum_{i\in\{1,3\}}\bG_{1i}(U) -\sum_{i\in\{1,3\}}\bG_{2i}(U) + \bB_{new}(U) \\ 
            &+  C\sum_{i\in\{1,3\}} \delta_i\sqrt{\delta_i}\exp(-C\delta_it) \int_{\mathbb{R}}\eta\left(U|\bar{U}\right)dx,
		\end{aligned}
        \end{align}
		where 
		\begin{align*}
			\begin{aligned}
				& \bG_{1i}(U) := \frac{R\theta^{[i]}\sigma^{[i]}}{2(v^{[i]})^2}  \int_{\mathbb{R}} |(a_i)_x| \phi_i \left[(v-\bar{v}) - \frac{u-\bar{u}}{\sigma^{[i]}}\right]^2  dx, \\
				& \bG_{2i}(U) := \frac{R\sigma^{[i]}}{2(\gamma-1)\theta^{[i]}}  \int_{\mathbb{R}}  |(a_i)_x|\phi_i \left[(\theta-\bth)+ \frac{(\gamma-1)\theta^{[i]}}{\sigma^{[i]}v^{[i]}}(u-\bar{u})\right]^2   dx, \\ 
				& \bB_{new} = \sum_{i\in\{1,3\}} \delta_i\int_{\mathbb{R}}(a_i)_x|(v-\bar{v},u-\bar{u},\theta-\bth)|^2 dx  +\sum_{i\in\{1,3\}} \int_{\mathbb{R}}(a_i)_x|(v-\bar{v},u-\bar{u},\theta-\bth)|^3dx  \\
				& \qquad\quad +C\int_{\mathbb{R}}(a_1)_x \left(|(\widetilde{v}_3-v^{[3]},\widetilde{\theta}_3-\theta^{[3]})| + |(v^D-v_*, \theta^D-\theta_*)|\right)\cdot|(v-\bar{v},u-\bar{u},\theta-\bth)|^2 dx  \\
				&  \qquad\quad + C\int_{\mathbb{R}}(a_3)_x\left( |(\widetilde{v}_1-v^{[1]},\widetilde{\theta}_1-\theta^{[1]})| + |(v^D-v^*, \theta^D-\theta^*)|\right)\cdot|(v-\bar{v},u-\bar{u},\theta-\bth)|^2 dx.
			\end{aligned}
		\end{align*}
To derive the estimate \eqref{est:BG}, several similar techniques are employed. For illustration, we present one representative computation, as the remaining terms can be handled in the same manner.
Making use of the relation $\phi_1 + \phi_3 = 1$ together with Lemma \ref{lem:cutoffinteraction}, we obtain
\begin{align*}
	\begin{aligned}
		& \frac{R\theta^{[i]}\sigma^{[i]}}{2(v^{[i]})^2}  \int_{\mathbb{R}} |(a_i)_x|  \left[(v-\bar{v}) - \frac{u-\bar{u}}{\sigma^{[i]}}\right]^2  dx \\
		= &  \frac{R\theta^{[i]}\sigma^{[i]}}{2(v^{[i]})^2}  \int_{\mathbb{R}} |(a_i)_x| \phi_i \left[(v-\bar{v}) - \frac{u-\bar{u}}{\sigma^{[i]}}\right]^2  dx +  \frac{R\theta^{[i]}\sigma^{[i]}}{2(v^{[i]})^2}  \int_{\mathbb{R}} |(a_i)_x| (1-\phi_i) \left[(v-\bar{v}) - \frac{u-\bar{u}}{\sigma^{[i]}}\right]^2  dx,  \\
		\leq &  \frac{R\theta^{[i]}\sigma^{[i]}}{2(v^{[i]})^2}  \int_{\mathbb{R}} |(a_i)_x| \phi_i \left[(v-\bar{v}) - \frac{u-\bar{u}}{\sigma^{[i]}}\right]^2  dx + C\delta_i\sqrt{\delta_i}\exp(-C\delta_it) \int_{\mathbb{R}}\eta\left(U|\bar{U}\right)dx. 
	\end{aligned}
\end{align*}
In the following estimates, the four favorable terms $\bG_{11},\bG_{12},\bG_{21},\bG_{22}$ will be utilized, and the unfavorable term $\bB_{new}$ can be controlled by combining these with
$\bD$ and $\mathcal{G}^S_i$, as demonstrated below.

    First, it follows from \eqref{eq:ax} and Lemma \ref{lem:cutoffinteraction} that
		\begin{align*}
			\begin{aligned}
				& \delta_i\int_{\mathbb{R}}(a_i)_x|(v-\bar{v},u-\bar{u},\theta-\bth)|^2 dx  \\
				\leq & \sqrt{\delta_i} \int_{\mathbb{R}}|(\widetilde{v}_i)_x |\phi_i(v-\bar{v},u-\bar{u},\theta-\bth)|^2 dx + \sqrt{\delta_i} \int_{\mathbb{R}}|(\widetilde{v}_i)_x|(1-\phi_i) |(v-\bar{v},u-\bar{u},\theta-\bth)|^2 dx  \\
				\leq & \sqrt{\delta_i} \mathcal{G}^S_i + C\sqrt{\delta_i} \delta^2_i\exp(-C\delta_it) \int_{\mathbb{R}}\eta\left(U|\bar{U}\right)dx,
			\end{aligned}
		\end{align*}
		where
		$$\mathcal{G}_i^{S}:=\int_{\mathbb{R}}(\widetilde{v}_i)_x|\phi_i(v-\bar{v},u-\bar{u},\theta-\bth)|^2 dx. $$
By virtue of Young’s inequality and the identity $\phi_1 + \phi_3 = 1$, it follows that
		\begin{align}\label{est:bad3}
			\begin{aligned}
				&  \int_{\mathbb{R}}(a_1)_x|(v-\bar{v},u-\bar{u},\theta-\bth)|^3dx   \\
				\leq &C \int_{\mathbb{R}} |(a_1)_x|\left|(v-\bar{v}) - \frac{u-\bar{u}}{\sigma^{[1]}}\right|^3  dx + C\int_{\mathbb{R}} |(a_1)_x||u-\bar{u}|^3 dx \\
                & + C\int_{\mathbb{R}}  |(a_1)_x| \left|(\theta-\bth)+ \frac{(\gamma-1)\theta^{[1]}}{\sigma^{[1]}v^{[1]}}(u-\bar{u})\right|^3   dx  \\
				\leq &  C\int_{\mathbb{R}} |(a_1)_x|\phi_1\left|(v-\bar{v}) - \frac{u-\bar{u}}{\sigma^{[1]}}\right|^3  dx + C\int_{\mathbb{R}} |(a_1)_x|(1-\phi_1)\left|(v-\bar{v}) - \frac{u-\bar{u}}{\sigma^{[1]}}\right|^3  dx \\ 
                &+  C\int_{\mathbb{R}}  |(a_1)_x| \phi_1\left|(\theta-\bth)+ C\frac{(\gamma-1)\theta^{[1]}}{\sigma^{[1]}v^{[1]}}(u-\bar{u})\right|^3   dx \\ 
                &+  C\int_{\mathbb{R}}  |(a_1)_x| (1-\phi_1)\left|(\theta-\bth)+ \frac{(\gamma-1)\theta^{[1]}}{\sigma^{[1]}v^{[1]}}(u-\bar{u})\right|^3   dx  \\
				& + C\int_{\mathbb{R}} |(a_1)_x|\phi_1|u-\bar{u}|^3 dx  + C\int_{\mathbb{R}} |(a_1)_x|(1-\phi_1)|u-\bar{u}|^3 dx .
			\end{aligned}
			\end{align}
By \eqref{apriori_small}, we have 
			\begin{align*}
				\begin{aligned}
					& \int_{\mathbb{R}} |(a_1)_x|\phi_1\left|(v-\bar{v}) - \frac{u-\bar{u}}{\sigma^{[1]}}\right|^3  dx + \int_{\mathbb{R}}  |(a_1)_x| \phi_1\left|(\theta-\bth)+ \frac{(\gamma-1)\theta^{[1]}}{\sigma^{[1]}v^{[1]}}(u-\bar{u})\right|^3   dx \\
					\leq & C\varepsilon(\bG_{11} + \bG_{21}) .
				\end{aligned}
			\end{align*}
Applying Young’s inequality together with \eqref{apriori_small} and Lemma \ref{lem:cutoffinteraction}, we obtain
		\begin{align*}
				\begin{aligned}
				\int_{\mathbb{R}} |(a_1)_x|(1-\phi_1)\left|(v-\bar{v}) - \frac{u-\bar{u}}{\sigma^{[1]}}\right|^3  dx 
				\leq & C\varepsilon \frac{1}{\sqrt{\delta_1}} \int_{\mathbb{R}} |(\widetilde{v}_1)_x|\phi_3\left(|v-\bar{v}|^2 + |u-\bar{u}|^2\right) dx   \\
				\leq &  C\varepsilon \frac{1}{\sqrt{\delta_1}}\sup_{t,x}\left(|(\widetilde{v}_1)_x|\phi_3\right) \int_{\mathbb{R}} \left(|v-\bar{v}|^2 +  |u-\bar{u}|^2\right) dx  \\
				\leq & C \varepsilon \delta_1\sqrt{\delta_1}\exp(-C\delta_1t) \int_{\mathbb{R}}\eta\left(U|\bar{U}\right)dx.
				\end{aligned}
			\end{align*}
			Similarly, we have 
	\begin{align*}
	\begin{aligned}
		\int_{\mathbb{R}}  |(a_1)_x| (1-\phi_1)\left|(\theta-\bth)+ \frac{(\gamma-1)\theta^{[1]}}{\sigma^{[1]}v^{[1]}}(u-\bar{u})\right|^3   dx			
		&\leq C \varepsilon\delta_1\sqrt{\delta_1}\exp(-C\delta_1t) \int_{\mathbb{R}}\eta\left(U|\bar{U}\right)dx, \\
		\int_{\mathbb{R}} |(a_1)_x|(1-\phi_1)|u-\bar{u}|^3 dx &\leq C \varepsilon\delta_1\sqrt{\delta_1}\exp(-C\delta_1t) \int_{\mathbb{R}}\eta\left(U|\bar{U}\right)dx.
	\end{aligned}
\end{align*}
Using \eqref{eq:ax} and Sobolev inequality, we have
\begin{align*}
	\begin{aligned}
		&\int_{\mathbb{R}} |(a_1)_x|\phi_1|u-\bar{u}|^3 dx   \\
	\leq  &C\frac{1}{\sqrt{\delta_1}} \|u-\bar{u}\|^2_{L^{\infty}(\mathbb{R})} \int_{\mathbb{R}} |(\widetilde{v}_1)_x|\phi_1|u-\bar{u}| dx  \\
	\leq & C\frac{1}{\sqrt{\delta_1}}|\|(u-\bar{u})_x\|_{L^{2}(\mathbb{R})}\|u-\bar{u}\|_{L^{2}(\mathbb{R})} \sqrt{\int_{\mathbb{R}}|(\widetilde{v}_1)_x||\phi_1(u-\bar{u})|^2 dx} \sqrt{\int_{\mathbb{R}}|(\widetilde{v}_1)_x|dx}  \\
	\leq & C\varepsilon \|(u-\bar{u})_x\|_{L^{2}(\mathbb{R})} \sqrt{\int_{\mathbb{R}}|(\widetilde{v}_1)_x||\phi_1(u-\bar{u})|^2 dx}   \\
	\leq & C\varepsilon \|(u-\bar{u})_x\|^2_{L^{2}(\mathbb{R})} + C\varepsilon\int_{\mathbb{R}}|(\widetilde{v}_1)_x||\phi_1(u-\bar{u})|^2 dx \\ \leq & C\varepsilon(\bold{D}_{u_1} + \mathcal{G}^S_1).
	\end{aligned}
\end{align*}
Plugging the above estimates into \eqref{est:bad3}, we deduce that
\begin{align*}
	&  \int_{\mathbb{R}}(a_1)_x|(v-\bar{v},u-\bar{u},\theta-\bth)|^3dx   \\
	\leq & C\varepsilon(\bG_{11} + \bG_{21} +\bold{D}_{u_1} + \mathcal{G}_1^S) + C \varepsilon\delta_1\sqrt{\delta_1}\exp(-C\delta_1t) \int_{\mathbb{R}}\eta\left(U|\bar{U}\right)dx.
\end{align*}
		Similarly, we have
		\begin{align*}
		&  \int_{\mathbb{R}}(a_3)_x|(v-\bar{v},u-\bar{u},\theta-\bth)|^3dx   \\
		\leq & C\varepsilon(\bG_{21} + \bG_{23} +\bold{D}_{u_1} + \mathcal{G}_3^S) + C \varepsilon\delta_3\sqrt{\delta_3}\exp(-C\delta_3t) \int_{\mathbb{R}}\eta\left(U|\bar{U}\right)dx.
		\end{align*}
Likewise, using the interpolation inequality, Lemmas \ref{lm:interaction}, and \ref{lem:cutoffinteraction}, we have
	\begin{align*}
	   & \quad  \int_{\mathbb{R}}(a_1)_x \left(|(\widetilde{v}_3-v^{[3]},\widetilde{\theta}_3-\theta^{[3]})| + |(v^D-v_*, \theta^D-\theta_*)|\right)\cdot|(v-\bar{v},u-\bar{u},\theta-\bth)|^2 dx  \\ 
	   & \leq  C(\delta_C + \delta_3)(\bG_{11} + \bG_{21})  +C \varepsilon\delta_1\sqrt{\delta_1}\exp(-C\delta_1t) \int_{\mathbb{R}}\eta\left(U|\bar{U}\right)dx \\
	   & \quad + \frac{C}{\sqrt{\delta_1}} \int_{\mathbb{R}}(\widetilde{v}_1)_x \left(|(\widetilde{v}_3-v^{[3]},\widetilde{\theta}_3-\theta^{[3]})| + |(v^D-v_*, \theta^D-\theta_*)|\right)|u-\bar{u}|^2 dx   \\
	   &   \leq C(\delta_C + \delta_3)(\bG_{11} + \bG_{21})  +C \varepsilon\delta_1\sqrt{\delta_1}\exp(-C\delta_1t) \int_{\mathbb{R}}\eta\left(U|\bar{U}\right)dx  \\ 
	   & \quad + C\frac{1}{\sqrt{\delta_1}}\|u-\bar{u}\|^2_{L^4({\mathbb{R}})}\bigg[\left\|(\widetilde{v}_1)_x(v^D-v_*, \theta^D-\theta_*)\right\|_{L^2(\mathbb{R}) }  +\left\|(\widetilde{v}_1)_x(\widetilde{v}_3-v^{[3]},\widetilde{\theta}_3-\theta^{[3]}) \right\|_{L^2(\mathbb{R})}\bigg]    \\
	   & \leq  C\delta_0(\bG_{11} + \bG_{21})  +C \varepsilon\delta_1\sqrt{\delta_1}\exp(-C\delta_1t) \int_{\mathbb{R}}\eta\left(U|\bar{U}\right)dx    \\
	   &\quad + C\frac{1}{\sqrt{\delta_1}}\|(u-\bar{u})_x\|^{\frac{1}{2}}_{L^2(\mathbb{R})}\|u-\bar{u}\|^{\frac{3}{2}}_{L^2(\mathbb{R})}\delta_1^{\frac{3}{2}}\left(\delta_3 e^{-C\delta_3 t}  + \delta_Ce^{-Ct}\right) \\
	   & \leq  C\delta_0(\bG_{11} + \bG_{21}) + C \varepsilon\delta_1\sqrt{\delta_1}\exp(-C\delta_1t) \int_{\mathbb{R}}\eta\left(U|\bar{U}\right)dx   +C\bD_{u_1}^{\frac{1}{4}}\varepsilon^{\frac{3}{2}}\delta_1\left(\delta_3 e^{-C\delta_3 t}  + \delta_Ce^{-Ct}\right) \\
	   & \leq C\delta_0(\bG_{11} + \bG_{21}) + C \varepsilon\delta_1\sqrt{\delta_1}\exp(-C\delta_1t) \int_{\mathbb{R}}\eta\left(U|\bar{U}\right)dx  + C\varepsilon\bD_{u_1} + \delta_1^{\frac{4}{3}}\left(\delta_3^{\frac{4}{3}} e^{-C\delta_3 t}  + \delta_C^{\frac{4}{3}}e^{-Ct}\right).
	\end{align*}
Similarly, we have
\begin{align*}
	\begin{aligned}
		& \int_{\mathbb{R}}(a_3)_x\left( |(\widetilde{v}_1-v^{[1]},\widetilde{\theta}_1-\theta^{[1]})| + |(v^D-v^*, \theta^D-\theta^*)|\right)\cdot|(v-\bar{v},u-\bar{u},\theta-\bth)|^2 dx  \\
		\leq & \ C\delta_0(\bG_{13} + \bG_{23}) + C \varepsilon\delta_3\sqrt{\delta_3} \exp(-C\delta_3t) \int_{\mathbb{R}}\eta\left(U|\bar{U}\right)dx \\ 
        & + C\varepsilon\bD_{u_1} + \delta_3^{\frac{4}{3}}\left(\delta_1^{\frac{4}{3}} e^{-C\delta_1 t}  + \delta_C^{\frac{4}{3}}e^{-Ct}\right).
	\end{aligned}
\end{align*}
	Therefore, we have
	\begin{align*}
		\bB_{new}  \leq & \ C(\sqrt{\delta_0}+ \varepsilon + \delta_0 )\left(\sum_{i\in\{1,3\}}\bG_{1i} + \sum_{i\in\{1,3\}}\bG_{2i}  + \bD_{u_1} +\sum_{i\in\{1,3\}}\mathcal{G}^S_i\right) \\
        & + C\sum_{i\in\{1,3\}}  \varepsilon\sqrt{\delta_i} \delta_i\exp(-C\delta_it) \int_{\mathbb{R}}\eta\left(U|\bar{U}\right)dx\\ 
        &+ \delta_1^{\frac{4}{3}}\left(\delta_3^{\frac{4}{3}} e^{-C\delta_3 t}  + \delta_C^{\frac{4}{3}}e^{-Ct}\right) +  \delta_3^{\frac{4}{3}}\left(\delta_1^{\frac{4}{3}} e^{-C\delta_1 t}  + \delta_C^{\frac{4}{3}}e^{-Ct}\right).  
	\end{align*}
	\noindent$\bullet$ (Estimate on  $-\frac{\delta_i}{2M_i}|\dot{\bX}_i|^2$): 
In order to bound $-\frac{\delta_i}{2M_i}|\dot{\bX}i|^2$, it suffices to control $\bold{Y}_{ij}$, as indicated by the decomposition \eqref{eq:decomposeX}.
	
We focus on the estimation of $-\frac{\delta_1}{2M_1}|\dot{\bX}_1|^2$ without loss of generality, since the arguments for the terms $-\frac{\delta_1}{2M_1}|\dot{\bX}_1|^2$ and $-\frac{\delta_3}{2M_3}|\dot{\bX}_3|^2$ are analogous.
By applying $\phi_1 + \phi_3 = 1$ and making the corresponding change of variables for $y_1$ and $w_1$, we deduce
	\begin{align*}
		\bold{Y}_{11}  = \int_{\mathbb{R}} a(\widetilde{u}_1)_x(u-\bar{u}) dx & = -\sigma_1\int_{\mathbb{R}} a(\widetilde{v}_1)_x\phi_1(u-\bar{u}) dx   -\sigma_1\int_{\mathbb{R}} a(\widetilde{v}_1)_x\phi_3(u-\bar{u}) dx   \\
		& = -\sigma_1\delta_1\int_{0}^{1}aw_1dy_1 - \sigma_1\int_{\mathbb{R}} a(\widetilde{v}_1)_x\phi_3(u-\bar{u}) dx,  
	\end{align*}
	Applying \eqref{est:sigma} and $|a-1| \leq C\sqrt{\delta_0}$, we have 
	\begin{equation}\label{y11}
	\left|\bold{Y}_{11} - \delta _1\sigma^{[1]}\int_{0}^{1} w_1dy_1 \right|
	\leq C\delta_1(\sqrt{\delta_0}+ \delta_0)\int_{0}^{1} w_1dy_1 + C\int_{\mathbb{R}} |(\widetilde{v}_1)_x|\phi_3(u-\bar{u}) dx.
	\end{equation}
	Note that
	 \begin{align*}
		\bold{Y}_{12}= &  \int_{\mathbb{R}} a\frac{(\widetilde{v}_1)_x\bar{p}}{\bar{v}}(v-\bar{v})dx \\
		=&  \int_{\mathbb{R}} a\frac{(\widetilde{v}_1)_x\bar{p}}{\bar{v}}\phi_1(v-\bar{v})dx+  \int_{\mathbb{R}} a\frac{(\widetilde{v}_1)_x\bar{p}}{\bar{v}}\phi_3(v-\bar{v})dx    \\
		= & \int_{\mathbb{R}} a\frac{(\widetilde{v}_1)_x\bar{p}}{\sigma^{[1]}\bar{v}}\phi_1(u-\bar{u})dx + \int_{\mathbb{R}} a\frac{(\widetilde{v}_1)_x\bar{p}}{\bar{v}}\phi_1\left(v-\bar{v}-\frac{u-\bar{u}}{\sigma^{[1]}}\right)dx +  \int_{\mathbb{R}} a\frac{(\widetilde{v}_1)_x\bar{p}}{\bar{v}}\phi_3(v-\bar{v})dx    \\
		=& \int_{0}^{1} a\frac{\bar{p}\delta_1}{\sigma^{[1]}\bar{v}}w_1dy_1 + \int_{\mathbb{R}} a\frac{(\widetilde{v}_1)_x\bar{p}}{\bar{v}}\phi_1\left(v-\bar{v}-\frac{u-\bar{u}}{\sigma^{[1]}}\right)dx +  \int_{\mathbb{R}} a\frac{(\widetilde{v}_1)_x\bar{p}}{\bar{v}}\phi_3(v-\bar{v})dx. 
	\end{align*}
From \eqref{est:pp} and \eqref{est:ax}, it follows that
	\begin{align}
		\begin{aligned}\label{y12}
		&\left|\bold{Y}_{12} -\frac{p^{[1]}\delta_1}{\sigma^{[1]}v^{[1]}} \int_{0}^{1}w_1dy_1 \right|  \\
		\leq & \ C\delta_1(\delta_0 +\sqrt{\delta_0}) \int_{\mathbb{R}} w_1dy_1 + C\sqrt{\delta_1}\int_{\mathbb{R}} |(a_1)_x|\phi_1\left|v-\bar{v}-\frac{u-\bar{u}}{\sigma^{[1]}}\right| dx +  \int_{\mathbb{R}} a\frac{(\widetilde{v}_1)_x\bar{p}}{\bar{v}}\phi_3(v-\bar{v})dx .   \\
		\end{aligned}
	\end{align}
In a similar manner, we obtain
	\begin{align}
		\begin{aligned}\label{y13}
		& \left|\bold{Y}_{13} - \frac{(\gamma-1)p^{[1]}\delta_1}{\sigma^{[1]}v^{[1]}} \int_{0}^{1} w_1dy_1 \right|  \\
		\leq & \ C\delta_1(\delta_0 +\sqrt{\delta_0}+ \varepsilon) \int_{0}^{1} |w_1|dy_1 + C\sqrt{\delta_1} \int_{\mathbb{R}} |(a_1)_x|\phi_1 \left[(\theta-\bth)+ \frac{(\gamma-1)\theta^{[1]}}{\sigma^{[1]}v^{[1]}}(u-\bar{u})\right]   dx\\
		& \quad  + \int_{\mathbb{R}} a\frac{R}{\gamma-1}\phi_3\frac{(\widetilde{\theta}_1)_x}{\bth}(\theta-\bth) dx,
		\end{aligned}
	\end{align}
    where we additionally used $|\theta- \theta^{[1]}| \leq C(\delta_0 + \varepsilon)$ and \eqref{est:vtheta}.
Therefore, using \eqref{eq:decomposeX}, \eqref{y11}, \eqref{y12} and \eqref{y13} together with $\sigma^{[1]} = \sqrt{\frac{\gamma p^{[1]}}{v^{[1]}}}$, we have 
\begin{align*}
	& \left|\dot{\bX}_1 + 2\sigma^{[1]}M_1 \int_{0}^{1} |w_1|dy_1 \right|    \\
	 \leq & \ (\delta_0 +\sqrt{\delta_0}+ \varepsilon) \int_{0}^{1} |w_1|dy_1 + \frac{C}{\sqrt{\delta_1}}\int_{\mathbb{R}} (a_1)_x\phi_1\left(v-\bar{v}-\frac{u-\bar{u}}{\sigma^{[1]}}\right)dx   \\
	& \quad+ \frac{C}{\sqrt{\delta_1}} \int_{\mathbb{R}} |(a_1)_x| \left[(\theta-\bth)+ \frac{(\gamma-1)\theta^{[1]}}{\sigma^{[1]}v^{[1]}}(u-\bar{u})\right]   dx\\
	& \quad  + \frac{C}{\delta_1}\int_{\mathbb{R}} |(\widetilde{v}_1)_x\phi_1||u-\bar{u}| dx +   \frac{C}{\delta_1}\int_{\mathbb{R}} |(\widetilde{v}_1)_x\phi_1||v-\bar{v}|dx +  \frac{C}{\delta_1}\int_{\mathbb{R}} |\phi_1(\widetilde{\theta}_1)_x||\theta-\bth| dx ,
\end{align*}
which yields
\begin{align*}
	& \left|2\sigma^{[1]}M_1\int_{0}^{1} |w_1|dy_1 - (-\dot{\bX}_1)  \right|^2    \\
	\leq & \ (\delta_0 +\sqrt{\delta_0}+ \varepsilon)^2 \int_{0}^{1} |w_1|^2dy_1 + C\frac{1}{\delta_1}\left(\int_{\mathbb{R}} |(a_1)_x| \phi_1\left(v-\bar{v}-\frac{u-\bar{u}}{\sigma^{[1]}}\right)dx \right)^2  \\
	& \quad+ C\frac{1}{\delta_1} \left(\int_{\mathbb{R}} |(a_1)_x| \phi_1 \left[(\theta-\bth)+ \frac{(\gamma-1)\theta^{[1]}}{\sigma^{[1]}v^{[1]}}(u-\bar{u})\right]  dx\right)^2\\
	& \quad  + \frac{C}{\delta^2_1}\left(\int_{\mathbb{R}} |(\widetilde{v}_1)_x|\phi_3|u-\bar{u}| dx\right)^2 +   \frac{C}{\delta^2_1}\left(\int_{\mathbb{R}} |(\widetilde{v}_1)_x|\phi_3|v-\bar{v}|dx\right)^2 +  \frac{C}{\delta^2_1}\left(\int_{\mathbb{R}} |(\widetilde{\theta}_1)_x|\phi_3|\theta-\bth| dx\right)^2 .
\end{align*}
This, in conjunction with the basic inequality $\frac{p^2}{2} - q^2 \le (p - q)^2$, yields
\begin{align*}
	2(\sigma^{[1]})^2M_1^2\left(\int_{0}^{1} |w_1|dy_1\right)^2 -|\dot{\bX}_1|^2  \leq \ & C(\delta_0 +\sqrt{\delta_0} + \varepsilon)^2 \int_{\mathbb{R}} |w_1|^2dy_1   \\
	& + \frac{C}{\sqrt{\delta_1}}(\bG_{11} + \bG_{21}) + C\delta_1\exp(-C\delta_1t) \int_{\mathbb{R}}\eta\left(U|\bar{U}\right)dx.
\end{align*}
Thus, 
\begin{align}
\begin{aligned}\label{estimateX1}
-\frac{\delta_1}{2M_1}|\dot{\bX}_1|^2  \leq & \ -(\sigma^{[1]})^2M_1\delta_1\left(\int_{0}^{1} w_1dy_1\right)^2  + C\delta_1 (\delta_0 +\sqrt{\delta_0}+ \varepsilon)^2 \int_{0}^{1} |w_1|^2dy_1   \\
& \ + C\sqrt{\delta_1}(\bG_{11} + \bG_{21}) + C\delta^2_1\exp(-C\delta_1t) \int_{\mathbb{R}}\eta\left(U|\bar{U}\right)dx.
\end{aligned}
\end{align}
	\noindent$\bullet$ (Estimate on  $\bB_1$): 
	Notice
	\begin{align}
		\begin{aligned}\label{eqsplitb1}
			\bold{B}_1(U)= &\sum\limits_{i\in \{1, 3\}}\int_{\mathbb{R}} a(\widetilde{v}_i)_x\left[ \frac{(\gamma+1)p^{[i]}\sigma^{[i]}}{2(v^{[i]})^2} |v-\bar{v}|^2 + \frac{\sigma^{[i]}R}{2v^{[i]}\theta^{[i]}}(\theta-\bth)^2-\frac{\sigma^{[i]}p^{[i]}}{2v^{[i]}\theta^{[i]}}(v-\bar{v})(\theta-\bth)\right] dx \\
			:= & \sum\limits_{i\in \{1, 3\}} \left(\bB^i_{11} + \bB^i_{12}+ \bB^i_{13}\right).
		\end{aligned}
	\end{align}
	Given that the estimates are the same for both $i=1$ and $i=3$, it suffices to treat the case $i=1$ as representative.
    First, note that
	\begin{align*}
		\begin{aligned}
			 \bB^1_{11} = \frac{(\gamma+1)p^{[1]}\sigma^{[1]}}{2(v^{[1]})^2} \int_{\mathbb{R}} a(\widetilde{v}_1)_x\phi_1^2 |v-\bar{v}|^2  dx + \frac{(\gamma+1)p^{[1]}\sigma^{[1]}}{2(v^{[1]})^2} \int_{\mathbb{R}} a(\widetilde{v}_1)_x(1-\phi_1^2) |v-\bar{v}|^2  dx 
		\end{aligned}
	\end{align*}
Applying Young’s inequality and using \eqref{eq:ax} for the first term, we deduce
	\begin{align*}
		\begin{aligned}
			& \frac{(\gamma+1)p^{[1]}\sigma^{[1]}}{2(v^{[1]})^2} \int_{\mathbb{R}} a(\widetilde{v}_1)_x\phi_1^2 |v-\bar{v}|^2  dx    \\
			= &  \frac{(\gamma+1)p^{[1]}\sigma^{[1]}}{2(v^{[1]})^2} \int_{\mathbb{R}} a(\widetilde{v}_1)_x\phi_1^2 \left|\left(v-\bar{v}-\frac{u-\bar{u}}{\sigma^{[1]}}\right)+ \frac{u-\bar{u}}{\sigma^{[1]}}\right|^2  dx \\
			\leq & \frac{(\gamma+1)p^{[1]}\sigma^{[1]}}{2(v^{[1]})^2}\left[1+\sqrt{\delta_0} + \delta_1^{\frac{1}{4}}\right] \int_{\mathbb{R}}|(\widetilde{v}_1)_x|\left|\frac{\phi_1(u-\bar{u})}{\sigma^{[1]}}\right|^2  \\
			& +C\delta_1^{-\frac{1}{4}} \int_{\mathbb{R}} |(\widetilde{v}_1)_x|\left|\phi_1 \left(v-\bar{v}-\frac{u-\bar{u}}{\sigma^{[1]}}\right)\right|^2 dx   \\ 
			\leq & \frac{(\gamma+1)p^{[1]}}{2(v^{[1]})^2\sigma^{[1]}}\left[1+\sqrt{\delta_0} + \delta_1^{\frac{1}{4}}\right]\delta_1 \int_{0}^{1}|w_1|^2dy_1 +  C\delta_1^\frac{1}{4}\bG_{11}.
		\end{aligned}
	\end{align*}
In the same way, the second term of $\bB^1_{11}$ is bounded by
	\begin{align*}
		\begin{aligned}
		 	\frac{(\gamma+1)p^{[1]}\sigma^{[1]}}{2(v^{[1]})^2} \int_{\mathbb{R}} a|(\widetilde{v}_1)_x|(1+\phi_1)\phi_3 |v-\bar{v}|^2  dx   
		 \leq &C \int_{\mathbb{R}} |(\widetilde{v}_1)_x|\phi_3 |v-\bar{v}|^2  dx   \\
		\leq & C\delta^2_1\exp(-C\delta_1t) \int_{\mathbb{R}}\eta\left(U|\bar{U}\right)dx.
		\end{aligned}
	\end{align*}
	Therefore, we have 
	\begin{align}
		\begin{aligned}\label{estb111}
		 \bB^1_{11} \leq  &\frac{(\gamma+1)p^{[1]}}{2(v^{[1]})^2\sigma^{[1]}}\left[1+\sqrt{\delta_0} + \delta_1^{\frac{1}{4}}\right]\delta_1 \int_{0}^{1}|w_1|^2dy_1 +  C\delta_1^{\frac{1}{4}}\bG_{11}  \\
		 & + C\delta^2_1\exp(-C\delta_1t) \int_{\mathbb{R}}\eta\left(U|\bar{U}\right)dx.
		 \end{aligned}
	\end{align}
	Similarly, we estimate 
	\begin{align}
		\begin{aligned}\label{estb112}
			\bB^1_{12} = & \ \frac{\sigma^{[1]}R}{2v^{[1]}\theta^{[1]}}\int_{\mathbb{R}} a(\widetilde{v}_1)_x\left| \theta-\bth + \frac{(\gamma-1)\theta^{[1]}}{\sigma^{[1]}v^{[1]}}(u-\bar{u}) - \frac{(\gamma-1)\theta^{[1]}}{\sigma^{[1]} v^{[1]}}(u-\bar{u}) \right|^2 dx  \\
			 \leq & \ \frac{(\gamma-1)^2p^{[1]}}{2\sigma^{[1]}(v^{[1]})^2}\left[1+\sqrt{\delta_0} + \delta_1^{\frac{1}{4}}\right]\delta_1 \int_{0}^{1}|w_1|^2dy_1 +  C\delta_1^{\frac{1}{4}}\bG_{21}  \\
			& \quad + C\delta^2_1\exp(-C\delta_1t) \int_{\mathbb{R}}\eta\left(U|\bar{U}\right)dx,
		\end{aligned}
	\end{align}
	and 
		\begin{align*}
		\begin{aligned}
			 \bB^1_{13} = &-\frac{\sigma^{[1]}p^{[1]}}{2v^{[1]}\theta^{[1]}}\int_{\mathbb{R}} a(\widetilde{v}_1)_x(v-\bar{v})(\theta-\bth)dx  \\
			\leq & \frac{\sigma^{[1]}p^{[1]}}{2v^{[1]}\theta^{[1]}}\int_{\mathbb{R}} a|(\widetilde{v}_1)_x|\phi_1^2\left|\left(v-\bar{v} -\frac{u-\bar{u}}{\sigma^{[1]}}\right)+\frac{u-\bar{u}}{\sigma^{[1]}} \right|  \\
		& \qquad\qquad \cdot	\left|\left(\theta-\bth+ \frac{(\gamma-1)\theta^{[1]}}{\sigma^{[1]}v^{[1]}}(u-\bar{u}\right)-\frac{(\gamma-1)\theta^{[1]}}{\sigma^{[1]}v^{[1]}}(u-\bar{u})\right|dx   \\
		& +  \frac{\sigma^{[1]}p^{[1]}}{2v^{[1]}\theta^{[1]}}\int_{\mathbb{R}} a|(\widetilde{v}_1)_x|\left(1-\phi_1^2\right)\left|v-\bar{v} \right| \left|\theta-\bth\right|dx  ,
		\end{aligned}
        \end{align*}
The first term admits the following estimate:
		\begin{align*}
		\begin{aligned}
			& \frac{\sigma^{[1]}p^{[1]}}{2v^{[1]}\theta^{[1]}}\int_{\mathbb{R}} a|(\widetilde{v}_1)_x|(\phi_1)^2\left|\left(v-\bar{v} -\frac{u-\bar{u}}{\sigma^{[1]}}\right)+\frac{u-\bar{u}}{\sigma^{[1]}} \right|  \\
			& \qquad\qquad \cdot	\left|\left(\theta-\bth+ \frac{(\gamma-1)\theta^{[1]}}{\sigma^{[1]}v^{[1]}}(u-\bar{u})\right)-\frac{(\gamma-1)\theta^{[1]}}{\sigma^{[1]}v^{[1]}}(u-\bar{u})\right|dx   \\
	        \leq & \  \frac{\sigma^{[1]}p^{[1]}}{2v^{[1]}\theta^{[1]}}\left[1+\sqrt{\delta_0} + \delta_1^{\frac{1}{4}}\right]\delta_1 \int_{0}^{1}|w_1|^2dy_1 + C\delta_1^{\frac{1}{4}}\bG_{11}+  C\delta_1^{\frac{1}{4}}\bG_{21} , \\
		\end{aligned}
	\end{align*}
The second term can be bounded as follows:
		\begin{align*}
		\begin{aligned}
			&  \frac{\sigma^{[1]}p^{[1]}}{2v^{[1]}\theta^{[1]}}\int_{\mathbb{R}} a|(\widetilde{v}_1)_x|\left(1-\phi_1^2\right)\left|v-\bar{v} \right| \left|\theta-\bth\right|dx 
            \leq C\delta^2_1\exp(-C\delta_1t) \int_{\mathbb{R}}\eta\left(U|\bar{U}\right)dx.
		\end{aligned}
	\end{align*}
	Therefore, we have 
		\begin{align}
		\begin{aligned}\label{estb113}
			 \bB^1_{13} \leq &  \frac{\sigma^{[1]}p^{[1]}}{2v^{[1]}\theta^{[1]}}\left[1+\sqrt{\delta_0} + \delta_1^{\frac{1}{4}}\right]\delta_1 \int_{0}^{1}|w_1|^2dy_1 +C\delta_1^{\frac{1}{4}}\bG_{11}+  C\delta_1^{\frac{1}{4}}\bG_{21}   \\
			&  + C\delta^2_1\exp(-C\delta_1t) \int_{\mathbb{R}}\eta\left(U|\bar{U}\right)dx.
		\end{aligned}
		\end{align}
Define $\alpha_1$ and $\alpha_3$ as the $O(1)$ constants given by
		\begin{align*}
			\begin{aligned}
				\alpha_1 := \frac{\gamma(\gamma+1)p^{[1]}}{2\sigma^{[1]}(v^{[1]})^2}  \qquad\text{and }\qquad	 \alpha_3 := \frac{\gamma(\gamma+1)p^{[3]}}{2\sigma^{[3]}(v^{[3]})^2} 
			\end{aligned}
		\end{align*}
Using \eqref{estb111}, \eqref{estb112}, and \eqref{estb113} in \eqref{eqsplitb1}, we deduce
		\begin{align}
			\begin{aligned}\label{est_B1}
				\bB_1 \leq  &\sum_{i\in\{1,3\}}\frac{\sigma^{[i]}p^{[i]}}{2v^{[i]}\theta^{[i]}}\left[1+\sqrt{\delta_0} + \delta_i^{\frac{1}{4}}\right]\delta_i \int_{0}^{1}|w_i|^2dy_i \\
				& +C\sum_{i\in\{1,3\}}\delta_i^{\frac{1}{4}}(\bG_{1i}+ \bG_{2i})   + C\sum_{i\in\{1,3\}}\delta^2_i\exp(-C\delta_it) \int_{\mathbb{R}}\eta\left(U|\bar{U}\right)dx. 
			\end{aligned}
		\end{align}
			\noindent$\bullet$ (Estimate on  $\bD$): 
			Split  
			\begin{align}
				\begin{aligned}\label{eq_D1D2}
					 & \bold{D}(U)=\int_{\mathbb{R}}  a\left[\frac{\mu}{v} |(u-\bar{u})_x|^2 + \frac{\kappa}{v\theta}\left|(\theta-\bar{\theta})_x\right|^2 \right]dx :=\widetilde{\bold{D}}_{u_1}(U) + \widetilde{\bold{D}}_{\theta_1}(U).
				\end{aligned}
			\end{align}
Note that $\widetilde{\bold{D}}_{u_1}(U)\sim \bold{D}_{u_1}(U)$ and $\widetilde{\bold{D}}_{\theta_1}(U)\sim \bold{D}_{\theta_1}(U)$, which are defined in \eqref{good_terms}.

First, noting that $1 \ge \phi_i \ge \phi_i^2$ for all $i$, we obtain
				\begin{align*}
				\begin{aligned}
					\widetilde{\bold{D}}_{u_1}(U)=& \int_{\mathbb{R}}  a(\phi_1 + \phi_3)\frac{\mu}{v} |(u-\bar{u})_x|^2 dx    \\
					 \geq & \sum_{i\in\{1,3\}} \int_{\mathbb{R}}  a\phi_i^2\frac{\mu}{v} |(u-\bar{u})_x|^2 dx  .
				\end{aligned}
			\end{align*}
By virtue of Young’s inequality, it follows that for any $\delta_* > 0$ sufficiently small,
				\begin{align*}
				\begin{aligned}
					& \int_{\mathbb{R}}  a\frac{\mu}{v}\left|\pa_x \left(\phi_i (u-\bar{u})\right)\right|^2 dx   \\
					 \leq & \ (1+\delta_*)\int_{\mathbb{R}}  a\frac{\mu}{v}\phi_i^2 \left|(u-\bar{u})^2_x \right|dx  + \frac{C}{\delta_*}  \int_{\mathbb{R}}  a\frac{\mu}{v}(\pa_x \phi_i)^2 \left|(u-\bar{u})^2\right|dx,   \\
				\end{aligned}
			\end{align*}
	we have 
	\begin{align*}
		 - \widetilde{\bold{D}}_{u_1}(U) & \leq - \frac{1}{1+\delta_*} \sum_{i\in\{1,3\}}\int_{\mathbb{R}}  a\frac{\mu}{v}|\pa_x\left(\phi_i (u-\bar{u})\right)|^2 dx  +  \frac{C}{\delta_*}  \int_{\mathbb{R}}  a\frac{\mu}{v}(\pa_x \phi_i)^2 \left|(u-\bar{u})^2\right|dx   \\
		& := J_1 + J_2.
	\end{align*}
	
In order to represent $J_1$ through the variables $y_i$ and $w_i$, we utilize the following bound: 
	\begin{align}\label{Difussion}
\left|{1\over y_i(1-y_i)}\frac{\mu}{\widetilde{v}_i}\frac{dy_i}{dx} - \delta_i \alpha_i\frac{\mu R\gamma}{\mu R\gamma + \kappa(\gamma-1)^2}\right| \leq C\delta^2_i.
	\end{align}
We refer to \cite[Appendix B]{KVW-NSF}  for the proof of \eqref{Difussion}. 	

With 
	$\left\|a\frac{\widetilde{v}_i}{v} -1\right\|_{L^{\infty}({\mathbb{R}})} \leq C(\delta_0 +\varepsilon +  \sqrt{\delta_0} )$, this yields
	\begin{align*}
	   \begin{aligned}
	   	J_1&  = - \frac{1}{1+\delta_*} \sum_{i\in\{1,3\}}  \int_{0}^{1} a\frac{ \widetilde{v}_i}{v} \frac{\mu}{\widetilde{v}_i}\left({dy_i\over dx}\right)  \left|\pa_{y_i} w_i\right|^2 dy_i  \\
	   	& \leq -\sum_{i\in\{1,3\}}\left(1-C(\delta_0 + \varepsilon + \sqrt{\delta_0} + \delta_*)\right)\left(\delta_i \alpha_i\frac{\mu R\gamma}{\mu R\gamma + \kappa(\gamma-1)^2}-C\delta^2_i \right) \int_{0}^{1} y_i(1-y_i)\left|\pa_{y_i} w_i\right|^2 dy_i   \\
	   	& \leq -\sum_{i\in\{1,3\}}\alpha_i\frac{\mu R\gamma}{\mu R\gamma + \kappa(\gamma-1)^2}\left(1-C(\delta_0 + \varepsilon + \sqrt{\delta_0} + \delta_*)\right) \delta_i \int_{0}^{1} y_i(1-y_i)\left|\pa_{y_i} w_i\right|^2 dy_i. 
	   \end{aligned}
	\end{align*}
	In order to estimate $J_2$, we apply the following bound, which holds for each $i=1,3$,
	\[\left|\pa_x(\phi_i(t,x))\right|\leq \frac{4}{\sigma_3 -\sigma_1}\frac{1}{t}, \qquad \forall x \in \mathbb{R}, \qquad t\in (0, T].\]
	Thus,
	\[J_2 \leq \frac{C}{\delta_*t^2}\int_{\mathbb{R}}\eta\left(U|\bar{U}\right)dx. \]
Consequently, for any $\delta_* \in (0,1)$ taken small enough, it follows that
	\begin{align}
		\begin{aligned}\label{est_D1}
			-\widetilde{\bold{D}}_{u_1} \leq & -\sum_{i\in\{1,3\}} \alpha_i\frac{\mu R\gamma}{\mu R\gamma + \kappa(\gamma-1)^2}\left(1-C(\delta_0 + \varepsilon + \sqrt{\delta_0}+ \delta_*)\right) \delta_i\int_{0}^{1} y_i(1-y_i)\left|\pa_{y_i} w_i\right|^2 dy_i   \\
			& +  \frac{C}{\delta_*t^2}\int_{\mathbb{R}}\eta\left(U|\bar{U}\right)dx.
		\end{aligned}
	\end{align}
	In similar manner, using the change of variable, we have 
		\begin{align}
		\begin{aligned}\label{est_D2}
			-\widetilde{\bold{D}}_{\theta_1} \leq & - \sum_{i\in\{1, 3\}}\alpha_i\frac{\mu R\gamma}{\mu R\gamma + \kappa(\gamma-1)^2}\frac{\kappa}{\mu\theta^{[i]}}\left(1-C(\delta_0 + \varepsilon + \sqrt{\delta_0} + \delta_*)\right) \delta_i\int_{0}^{1} y_i(1-y_i)\left|\pa_{y_i} \left(\phi_i(\theta-\bth)\right)\right|^2 dy_i  \\
			 & + \frac{C}{\delta_*t^2}\int_{\mathbb{R}}\eta\left(U|\bar{U}\right)dx.
		\end{aligned}
	\end{align}
Our next goal is to utilize the Poincaré-type inequality stated in Lemma \ref{Poincare}, together with \eqref{estimateX1}, in order to absorb the leading unfavorable term $\bB_1$ into the diffusive component $\bD$.
However, recalling that
\[1>\frac{\mu R\gamma}{\mu R \gamma+ \kappa(\gamma-1)^2} \rightarrow 0 \quad \text{as}\quad \gamma \rightarrow \infty,\]
and combining \eqref{est_B1} with \eqref{est_D1}, we find that the dissipation $\widetilde{\bold{D}}{u_1}$ alone cannot effectively dominate $\bB_1$.
To remedy this difficulty, extract an additional favorable term involving $w_i$ from $\widetilde{\bold{D}}_{\theta_1}$, as detailed below.

To begin with, we make use of Lemma \ref{Poincare} along with
	\[\int_{0}^{1} \left|w_i-\bar{w}_i\right|^2 dy_i = \int_{0}^{1} w_i^2 dy_i- \bar{w}_i^2, \qquad \bar{w}_i := \int_{0}^{1} w_i dy_i,\]
	we have 
	\begin{align*}
		\begin{aligned}
			- \bD  \leq & -\sum_{i\in\{1,3\}} 2\alpha_i\frac{\mu R\gamma}{\mu R\gamma + \kappa(\gamma-1)^2}\left(1-C(\delta_0 + \varepsilon +\sqrt{\delta_0} + \delta_*)\right) \delta_i \left[\int_{0}^{1} w_i^2 dy_i- \bar{w}_i^2\right]  \\
			&   - \sum_{i\in\{1,3\}} 2 \alpha_i\frac{\mu R\gamma}{\mu R\gamma + \kappa(\gamma-1)^2}\frac{\kappa}{\mu\theta^{[i]}}\left(1-C(\delta_0 + \varepsilon + \sqrt{\delta_0}+ \delta_*)\right) \delta_i  \\
			& \qquad \qquad \cdot \left[\int_{0}^{1}  \left(\phi_i(\theta-\bth)\right)^2 dy_i - \left(\int_{0}^{1}  \phi_i(\theta-\bth) dy_i\right)^2 \right] \\
			&+ \frac{C}{\delta_*t^2}\int_{\mathbb{R}}\eta\left(U|\bar{U}\right)dx.
		\end{aligned}
	\end{align*}
	Observe that Young's inequality yields
	\begin{align*}
		\begin{aligned}
		   \int_{0}^{1}  \left|\phi_i(\theta-\bth)\right|^2 dy_i =&   \int_{0}^{1}  \left|\phi_i\left(\theta-\bth+ \frac{(\gamma-1)\theta^{[i]}}{\sigma^{[i]} v^{[i]}}  (u-\bar{u})\right) -\frac{(\gamma-1)\theta^{[i]}}{\sigma^{[i]}v^{[i]}}(u-\bar{u}) \right|^2 dy_i  \\
		    \geq&  \left(\frac{(\gamma-1)\theta^{[i]}}{\sigma^{[i]}v^{[i]}}\right)^2\left(1-\delta_i^{\frac{1}{4}}\right)\int_{0}^{1}|w_i|^2 dy_i  \\
		   &  - \frac{C}{\delta_i^{\frac{1}{4}}} \int_{0}^{1}  \left[\phi_i\left(\theta-\bth+ \frac{(\gamma-1)\theta^{[i]}}{\sigma^{[i]} v^{[i]}}(u-\bar{u})\right)\right]^2 dy_i, 
		\end{aligned}
	\end{align*}
	and  
	\begin{align*}
		\begin{aligned}
			 \left(\int_{0}^{1}  \phi_i(\theta-\bth) dy_i\right)^2  \leq & \ 2\left[\int_{0}^{1}\frac{(\gamma-1)\theta^{[i]}}{\sigma^{[i]}v^{[i]}}\phi_i(u-\bar{u})dy_i\right]^2 \\ 
             &+ 2\left[\int_{0}^{1}\phi_i\left((\theta-\bth) + \frac{(\gamma-1)\theta^{[i]}}{\sigma^{[i]}v^{[i]}}\phi_i(u-\bar{u})\right)dy_i\right]^2,
		\end{aligned}
	\end{align*}

combining with $ \sigma^{[i]}=\frac{\sqrt{\gamma R\theta^{[i]}}}{v^{[i]}}$, \eqref{eq:ax} and \eqref{eq:dy} yield
	
	\begin{align*}
		\begin{aligned}
			- \bD \leq &  -\sum_{i\in\{1,3\}}2\alpha_i\left(1-C(\delta_0 + \varepsilon + \sqrt{\delta_0}+ \delta_*)\right) \delta_i\frac{\mu R\gamma}{\mu R\gamma + \kappa(\gamma-1)^2} \left(1 +\frac{\kappa(\gamma-1)^2}{\mu R\gamma}\right)  \cdot \left(1-\delta_i^{\frac{1}{4}}\right)\int_{0}^{1} |w_i|^2 dy_i   \\
			& + \sum_{i\in\{1,3\}} 2\alpha_i\left(1-C(\delta_0 + \varepsilon + \sqrt{\delta_0} + \delta_*)\right) \delta_i\left[1+ 2\frac{\kappa}{\mu \theta^{[i]}}\left(\frac{(\gamma-1)\theta^{[i]}}{\sigma^{[i]}v^{[i]}}\right)^2\right]\bar{w}_i^2  \\
			& + \sum_{i\in\{1,3\}}4\alpha_i\left(1-C(\delta_0 + \varepsilon + \sqrt{\delta_0} + \delta_*)\right) \delta_i\frac{\kappa}{\mu \theta^{[i]}} \int_{0}^{1}  \left[\phi_i\left(\theta-\bth+ \frac{(\gamma-1)\theta^{[i]}}{\sigma^{[i]}v^{[i]}}(u-\bar{u})\right)\right]^2 dy_i  \\
			& + \frac{C}{\delta_*t^2}\int_{\mathbb{R}}\eta\left(U|\bar{U}\right)dx\\
			\leq &   -\sum_{i\in\{1,3\}}2\alpha_i\left(1-C(\delta_0 + \varepsilon + \sqrt{\delta_0}+ \delta_*) -\delta_i^{\frac{1}{4}}\right) \delta_i
		\int_{0}^{1} \left|w_i\right|^2 dy_i   \\
		& + \sum_{i\in\{1,3\}}2\alpha_i\delta_i\left(1+ \frac{2\kappa(\gamma-1)^2}{\mu R \gamma}\right)\bar{w}_i^2  + C\sum_{i\in\{1,3\}}\delta_i^{\frac{1}{4}}\bG_{2i} + \frac{C}{\delta_*t^2}\int_{\mathbb{R}}\eta\left(U|\bar{U}\right)dx.
		\end{aligned}
		\end{align*}
	\noindent$\bullet$ (Summary): 	
Since $\delta_i$, $\delta_0$, and $\varepsilon$ are sufficiently small, it follows that
	\begin{align*}
		\begin{aligned}
			& -\sum_{i\in\{1,3\}}\frac{\delta_i}{2M_i}|\dot{\bold{X}_i}|^{2} + \bold{B}_1 + \bold{B}_2 -\bold{G} -\frac{3}{4}\bold{D}  \\
			\leq & \ C(\delta_0 + \varepsilon + \sqrt{\delta_0})\sum_{i\in\{1,3\}} \mathcal{G}_i^S  + \delta_1^{\frac{4}{3}}\left(\delta_3^{\frac{4}{3}} e^{-C\delta_3 t}  + \delta_C^{\frac{4}{3}}e^{-Ct}\right) +  \delta_3^{\frac{4}{3}}\left(\delta_1^{\frac{4}{3}} e^{-C\delta_1 t}  + \delta_C^{\frac{4}{3}}e^{-Ct}\right)  \\
			& -\sum_{i\in\{1,3\}}\frac{1}{4}\alpha_i \delta_i\int_{0}^{1} \left|w_i\right|^2 dy_i - \sum_{i\in\{1,3\}} \frac{1}{2}\left(\bG_{1i} +\bG_{2i}  \right) \\
            &-  \sum_{i\in\{1,3\}}(\sigma^{[i]})^2M_i\delta_i\bar{w}_i^2 + \sum_{i\in\{1,3\}}\frac{3}{2}\alpha_i\delta_i\left(1+ \frac{2\kappa(\gamma-1)^2}{\mu R \gamma}\right)\bar{w}_i^2   \\
			& + \sum_{i\in\{1,3\}}C\delta^2_i\exp(-C\delta_it) \int_{\mathbb{R}}\eta\left(U|\bar{U}\right)dx+ \frac{C}{\delta_*t^2}\int_{\mathbb{R}}\eta\left(U|\bar{U}\right)dx.
		\end{aligned}
	\end{align*}
Let the $O(1)$ constants $M_i$ be taken as defined in \eqref{def:shift}, namely,
	\[M_i = \frac{3}{2}\frac{\alpha_i}{(\sigma^{[i]})^2} \left(1+ \frac{2\kappa(\gamma-1)^2}{\mu R \gamma}\right),\]
	we have 
		\begin{align*}
		\begin{aligned}
			& -\sum\limits_{i\in \left\{1, 3\right\}}\frac{\delta_i}{2M_i}|\dot{\bold{X}_i}|^{2} + \bold{B}_1 + \bold{B}_2 -\bold{G} -\frac{3}{4}\bold{D}  \\
			 \leq & \ C(\delta_0 + \varepsilon + \sqrt{\delta_0} )\sum_{i\in\{1,3\}} \mathcal{G}_i^S  + \delta_1^{\frac{4}{3}}\left(\delta_3^{\frac{4}{3}} e^{-C\delta_3 t}  + \delta_C^{\frac{4}{3}}e^{-Ct}\right)  +  \delta_3^{\frac{4}{3}}\left(\delta_1^{\frac{4}{3}} e^{-C\delta_1 t}  + \delta_C^{\frac{4}{3}}e^{-Ct}\right) \\
			& -\sum_{i\in\{1,3\}}\frac{1}{4}\alpha_i \delta_i\int_{0}^{1} \left|w_i\right|^2 dy_i - \sum_{i\in\{1,3\}} \frac{1}{2}\left(\bG_{1i} +\bG_{2i}  \right)  \\
            & +  \left(\sum_{i\in\{1,3\}}C\delta^2_i\exp(-C\delta_it) + \frac{C}{\delta_*t^2}\right)\int_{\mathbb{R}}\eta\left(U|\bar{U}\right)dx.
		\end{aligned}
	\end{align*}
Finally, by
\[\delta_i\int_{0}^{1} w_i^2 dy_i = \int_{\mathbb{R}} \left|(\widetilde{v}_i)_x\right| \left|\phi_i (u-\bar{u})\right|^2 dx\]
and 
\begin{align*}
	\begin{aligned}
		& \int_{\mathbb{R}} \left|(\widetilde{v}_i)_x\right| \left|\phi_i (v-\bar{v})\right|^2 dx   \\
		 \leq & \ 2\int_{\mathbb{R}} \left|(\widetilde{v}_i)_x\right| \left|\phi_i \left(u-\bar{u}-\frac{u-\bar{u}}{\sigma^{[i]}}\right)\right|^2 dx + 2\int_{\mathbb{R}} \left|(\widetilde{v}_i)_x\right| \left|\phi_i \frac{u-\bar{u}}{\sigma^{[i]}}\right|^2 dx  \\
		\leq & \ C\sqrt{\delta_i}\bG_{1i}+ C\int_{\mathbb{R}} \left|(\widetilde{v}_i)_x\right| \left|\phi_i (u-\bar{u})\right|^2 dx ,
	\end{aligned}
\end{align*}
and 
\begin{align*}
		\begin{aligned}
			& \int_{\mathbb{R}} \left|(\widetilde{v}_i)_x\right| \left|\phi_i (\theta-\bth)\right|^2 dx   \\
			\leq & \ 2\int_{\mathbb{R}} \left|(\widetilde{v}_i)_x\right| \left|\phi_i \left(\theta-\bth-\frac{(\gamma-1)\sigma^{[i]}}{\sigma^{[i]}v^{[i]}}(u-\bar{u})\right)\right|^2 dx + C\int_{\mathbb{R}} \left|(\widetilde{v}_i)_x\right| \left|\phi_i (u-\bar{u})\right|^2 dx  \\
			 \leq & \ C\sqrt{\delta_i}\bG_{2i}+ C\int_{\mathbb{R}} \left|(\widetilde{v}_i)_x\right| \left|\phi_i (u-\bar{u})\right|^2 dx ,
		\end{aligned}
	\end{align*}
we have 

\begin{align*}
	\begin{aligned}
		& -\sum\limits_{i\in \left\{1, 3\right\}}\frac{\delta_i}{2M_i}|\dot{\bold{X}_i}|^{2}+ \bold{B}_1 + \bold{B}_2 -\bold{G} -\frac{3}{4}\bold{D}  \\
		 \leq & \ C\sum_{i\in\{1,3\}}\int_{\mathbb{R}} \left|(\widetilde{v}_i)_x\right| \left|\phi_i (v-\bar{v}, u-\bar{u},\theta-\bth )\right|^2 dx  + \delta_1^{\frac{4}{3}}\left(\delta_3^{\frac{4}{3}} e^{-C\delta_3 t}  + \delta_C^{\frac{4}{3}}e^{-Ct}\right) \\ 
         & +  \delta_3^{\frac{4}{3}}\left(\delta_1^{\frac{4}{3}} e^{-C\delta_1 t}  + \delta_C^{\frac{4}{3}}e^{-Ct}\right) +  \left(\sum_{i\in\{1,3\}}C\delta^2_i\exp(-C\delta_it) + \frac{C}{\delta_*t^2}\right)\int_{\mathbb{R}}\eta\left(U|\bar{U}\right)dx.
	\end{aligned}
\end{align*}
\subsection{Proof of Lemma \ref{lm_zeroorder}}
From \eqref{est_RE} and \eqref{eq_XY}, we first deduce that
\begin{align*}
		{d \over dt} \int_{\mathbb{R}}a\bar{\theta}\eta\left(U|\bar{U}\right)dx 
	\leq &  -\sum_{i\in\{1,3\}}\frac{\delta_i}{2M_i}|\dot{\bold{X}}_i|^{2} + \bold{B}_1 + \bold{B}_2 -\bold{G} -\frac{3}{4}\bold{D}   \\
	&  -\sum_{i\in\{1,3\}}\frac{\delta_i}{2M_i}|\dot{\bold{X}}_i|^{2} + \sum_{i\in\{1,3\}}\dot{\bold{X}}_i\sum_{j=4}^{6}\bold{Y}_{ij}+ \sum_{i=3}^{6}\bold{B}_i + \bold{S}_1 + \bold{S}_2  -\frac{1}{4}\bold{D}(U).
\end{align*}
Using Lemma \ref{lma_part1} in conjunction with Young’s inequality, it follows that
\begin{align}\label{est_total}
	\begin{aligned}
		 	{d \over dt} \int_{\mathbb{R}}a\bar{\theta}\eta\left(U|\bar{U}\right)dx  \leq&  C\sum_{i\in\{1,3\}}\mathcal{G}_i^{S}+ \delta_1^{\frac{4}{3}}\left(\delta_3^{\frac{4}{3}} e^{-C\delta_3 t}  + \delta_C^{\frac{4}{3}}e^{-Ct}\right) +  \delta_3^{\frac{4}{3}}\left(\delta_1^{\frac{4}{3}} e^{-C\delta_1 t}  + \delta_C^{\frac{4}{3}}e^{-Ct}\right)  \\
		& +  \left(\sum_{i\in\{1,3\}}C\delta^2_i\exp(-C\delta_it) + \frac{C}{\delta_*t^2}\right)\int_{\mathbb{R}}\eta\left(U|\bar{U}\right)dx   -\sum_{i\in\{1,3\}}\frac{\delta_i}{4M_i}|\dot{\bold{X}_i}|^{2}   \\
		&+  \sum_{i\in\{1,3\}}\frac{C}{\delta_i}\sum_{j=4}^{6}\bold{Y}_{ij}^2+ \sum_{i=3}^{6}\bold{B}_i + \bold{S}_1 + \bold{S}_2  -\frac{1}{4}\bold{D}(U).
	\end{aligned}
\end{align}
	\noindent$\bullet$ Estimates on the terms $\bold{Y}_{ij}(j=4,5,6)$: 	
From \eqref{eq:Phiexpan} and \eqref{est:vtheta}, we obtain
\begin{align*}
	\begin{aligned}
		\left|(\bold{Y}_{i4} , 	\bold{Y}_{i5})\right|&  \leq C \int_{\mathbb{R}} |(\widetilde{v}_i)_x|\left|(v-\bar{v}, \theta-\bth)\right|^2 dx  \\
		&  \leq C \int_{\mathbb{R}} |(\widetilde{v}_i)_x|\left| \phi_i(v-\bar{v}, \theta-\bth)\right|^2 dx  + C \int_{\mathbb{R}} |(\widetilde{v}_i)_x|\left|(1-\phi_i^2)(v-\bar{v}, \theta-\bth)\right|^2 dx   \\
		& \leq C\left(\mathcal{G}_i^S + \delta^2_i\exp(-C\delta_it)  \int_{\mathbb{R}}\eta\left(U|\bar{U}\right)dx\right).
	\end{aligned}
\end{align*}
Furthermore, Lemma \ref{lem:vs} together with \eqref{apriori_small} implies that
\[	\left|(\bold{Y}_{i4} , 	\bold{Y}_{i5})\right| \leq C\delta_i^2\varepsilon^2 .\]

Thus,
\[	\frac{C}{\delta_i}\left|(\bold{Y}_{i4} , 	\bold{Y}_{i5})\right|^2 \leq C\delta_i\varepsilon^2\left(\mathcal{G}_i^S + \delta^2_i\exp(-C\delta_it)  \int_{\mathbb{R}}\eta\left(U|\bar{U}\right)dx\right).\]

Similarly, we have 
\begin{align*}
	\begin{aligned}
		\frac{C}{\delta_i} |\bold{Y}_{i6}|^2   & \leq \frac{C}{\delta_i} \left( \int_{\mathbb{R}} |(a_i)_x|\left|(v-\bar{v}, u-\bar{u}, \theta-\bth)\right|^2 dx \right)^2  \\
		& \leq \frac{C}{\delta_i^2} \left( \int_{\mathbb{R}} |(\widetilde{v}_i)_x|\left|(v-\bar{v}, u-\bar{u}, \theta-\bth)\right|^2 dx \right)^2  \\
		& \leq C\left\| (v-\bar{v}, u-\bar{u}, \theta-\bth)\right\|^2_{L^2(\mathbb{R})}\int_{\mathbb{R}} |(\widetilde{v}_i)_x|\left|(v-\bar{v}, u-\bar{u}, \theta-\bth)\right|^2 dx    \\
		& \leq C\varepsilon^2\int_{\mathbb{R}} |(\widetilde{v}_i)_x|\left|\phi_i(v-\bar{v}, u-\bar{u}, \theta-\bth)\right|^2 dx  \\ 
        & \quad \ + \int_{\mathbb{R}} |(\widetilde{v}_i)_x|(1-\phi_i^2)\left|(v-\bar{v}, u-\bar{u}, \theta-\bth)\right|^2 dx  \\
		& \leq C\varepsilon^2\left(\mathcal{G}_i^S + \delta^2_i\exp(-C\delta_it)  \int_{\mathbb{R}}\eta\left(U|\bar{U}\right)dx\right).
	\end{aligned}
\end{align*}
In short,
\begin{align}\label{est_Yij}
		\frac{C}{\delta_i} \sum_{j=4}^{6}|\bold{Y}_{ij}|^2  \leq  C\varepsilon\left(\mathcal{G}_i^S + \delta^2_i\exp(-C\delta_it)  \int_{\mathbb{R}}\eta\left(U|\bar{U}\right)dx\right).
\end{align}

	\noindent$\bullet$ (Estimates on the terms $\bold{B}_{i} \ (i=3, 4, 5, 6)$): 	

First, we split $\bold{B}_3 $ as
\begin{align*}
     \bold{B}_3 = & -\int_{\mathbb{R}}  a_x\left[\mu \left(\frac{u_x}{v}-\frac{\bar{u}_x}{\bar{v}_x}\right)(u-\bar{u}) +\kappa\frac{\theta -\bar{\theta}}{\theta}\left(\frac{\theta_x}{v}-\frac{\bar{\theta}_x}{\bar{v}}\right)\right]dx  \\
   =&  - \sum_{i\in\{1,3\}}\int_{\mathbb{R}} (a_i)_x\left[\mu \left(\frac{u_x}{v}-\frac{\bar{u}_x}{\bar{v}_x}\right)(u-\bar{u}) +\kappa\frac{\theta -\bar{\theta}}{\theta}\left(\frac{\theta_x}{v}-\frac{\bar{\theta}_x}{\bar{v}}\right)\right]dx :=   \sum_{i\in\{1,3\}}\bold{B}_{3i} .
\end{align*}
Then, we have
  \begin{align*}
  	\begin{aligned}
  	    \bold{B}_{3i}  \leq C&  \int_{\mathbb{R}} |( a_i)_x||u-\bar{u}|\left(|(u-\bar{u})_x| + |\bar{u}_x||v-\bar{v}|\right) dx \\
  	  &   + \int_{\mathbb{R}} |( a_i)_x||\theta-\bth|\left(|(\theta-\bth)_x| + |\bth_x||v-\bar{v}|\right) dx \\
  	  \leq & \frac{1}{80} \bD_{u_1}+  \int_{\mathbb{R}} |( a_i)_x|^2\left(|u-\bar{u}|^2 + |\theta-\bth|^2\right)dx  + C \int_{\mathbb{R}} \left(|\bar{u}_x|^2 + |\bth_x|^2\right)|v-\bar{v}|^2 dx  \\
  	  \leq &  \frac{1}{80} \bD_{u_1} +  C\delta_i \int_{\mathbb{R}} |( \widetilde{v}_i)_x|\left(|u-\bar{u}|^2 + |\theta-\bth|^2\right)dx  + C\int_{\mathbb{R}} \left(|\bar{u}_x|^2 + |\bth_x|^2\right)|v-\bar{v}|^2 dx   \\
  	  \leq& \frac{1}{80} \bD_{u_1} +  C\delta_i \left(\mathcal{G}_i^S + \delta^2_i\exp(-C\delta_it)  \int_{\mathbb{R}}\eta\left(U|\bar{U}\right)dx\right) + C\int_{\mathbb{R}} \left(|\bar{u}_x|^2 + |\bth_x|^2\right)|v-\bar{v}|^2 dx .
  	\end{aligned}
  \end{align*}
Using \eqref{est:vcdproperty} and \eqref{est:shock_prop} yield
\begin{align*}
& 	\int_{\mathbb{R}} \left(|\bar{u}_x|^2 + |\bth_x|^2\right)|v-\bar{v}|^2 dx   \\
	 \leq & \ 	\int_{\mathbb{R}} \left(|(\widetilde{u}_1)_x|^2  + |u^D_x|^2 + |(\widetilde{u}_3)_x|^2+|(\widetilde{\theta}_1)_x|^2 + |\theta^D_x|^2 + |(\widetilde{\theta}_3)_x|^2 \right)|v-\bar{v}|^2 dx  \\
	\leq & \ C\sum_{i\in\{1,3\}} \left(\delta^2_i \mathcal{G}_i^S + \delta^2_i\exp(-C\delta_it)  \int_{\mathbb{R}}\eta\left(U|\bar{U}\right)dx \right)+  C\delta_C(1+t)^{-1} \int_{\mathbb{R}}e^{-\frac{2C_1|x|^2}{1+t}}|v-\bar{v}|^2 dx.
\end{align*}

\noindent Thus, we have 
\begin{align*}
\bB_3 \leq & \frac{1}{40} \bD_{u_1} + C\sum_{i\in\{1,3\}} (\delta_i  + \delta^2_i) \mathcal{G}_i^S + C\sum_{i\in\{1,3\}}\delta^2_i\exp(-C\delta_it)  \int_{\mathbb{R}}\eta\left(U|\bar{U}\right)dx \\ 
& +  C\delta_C(1+t)^{-1} \int_{\mathbb{R}}e^{-\frac{2C_1|x|^2}{1+t}}|v-\bar{v}|^2 dx.
\end{align*}
Likewise, we have 
\begin{align*}
	\bB_4 \leq& \int_{\mathbb{R}}\left[|\bar{u}_x||v-\bar{v}||(u-\bar{u})_x| + |\theta-\bth|\left(|(\theta-\bth)_x| + |\bth_x|\right)\left(|(\theta-\bth)_x + |\bth_x|||v-\bar{v}|\right)\right]  dx\\
	&+ \int_{\mathbb{R}}\left[|(\theta-\bth)_x||\bth_x||v-\bar{v}| + |\theta-\bth|\left(|(u-\bar{u})_x| ^2  + |\bar{u}_x||(u-\bar{u})_x| +  |\bar{u}_x|^2||v-\bar{v}| \right)\right]  dx,  \\
	\leq  & \frac{1}{80} \bD_{u_1}+ C\sum_{i\in\{1,3\}}  \delta^2_i \mathcal{G}_i^S + \sum_{i\in\{1,3\}}\delta^2_i\exp(-C\delta_it)  \int_{\mathbb{R}}\eta\left(U|\bar{U}\right)dx \\ 
    & +  C\delta_C(1+t)^{-1} \int_{\mathbb{R}}e^{-\frac{2C_1|x|^2}{1+t}}|(v-\bar{v},\theta-\bth)|^2 dx.
\end{align*}
To estimate $\bB_5$ and $\bB_6$ in the final step, we first demonstrate that $\bB_5$ is bounded by $C\bB_6$.

Note that from Lemma \ref{lem:vcd} and \ref{lem:vs},
\begin{equation*}
\left|\left(\frac{\bar{\theta}_x}{\bar{v}}\right)_x\right|, \left|\bar{u}_x^2\right| \leq C\delta_0(|(\widetilde{v}_1)_x|+|(\widetilde{v}_3)_x|)+C\delta_C(1+t)^{-1}e^{-\frac{C_1x^2}{1+t}}.
\end{equation*}
Then 
\begin{align*}
\bB_5 \leq& C\delta_0\int_{\mathbb{R}} |\theta-\bar{\theta}|^2 \,dx + C\delta_C(1+t)^{-1}\int_{\mathbb{R}} ae^{-\frac{C_1|x|^2}{1+t}}|\theta-\bth|^2 dx\leq C\bB_6.
\end{align*}
Using \eqref{est:shock_prop}, we have 
\begin{align*}
      	 \bold{B}_6 =& C\delta_C(1+t)^{-1} \int_{\mathbb{R}} ae^{-\frac{C_1|x|^2}{1+t}}|(v-\bar{v},\theta-\bth)|^2 dx   \\
      	&+  C(\delta_0 + \varepsilon)\int_{\mathbb{R}}\left(|(\widetilde{v}_1)_x| + |(\widetilde{v}_3)_x|\right)|(v-\bar{v},\theta-\bth)|^2  dx \\ 
      	\leq & C(\delta_0 + \varepsilon) \sum_{i\in\{1,3\}} \mathcal{G}_i^S + \sum_{i\in\{1,3\}}\delta^2_i\exp(-C\delta_it)  \int_{\mathbb{R}}\eta\left(U|\bar{U}\right)dx \\ 
        & +  C\delta_C(1+t)^{-1} \int_{\mathbb{R}}e^{-\frac{2C_1|x|^2}{1+t}}|(v-\bar{v},\theta-\bth)|^2 dx.
\end{align*}

Therefore, we have
\begin{align}\label{est_Bi}
	\begin{aligned}
		\sum_{i=3}^{6} \bB_i\leq & \frac{1}{20} \bD_{u_1} +C(\delta_0 + \sqrt{\delta_0} +\varepsilon) \sum_{i\in\{1,3\}}  \mathcal{G}_i^S + \sum_{i\in\{1,3\}}\delta^2_i\exp(-C\delta_it)  \int_{\mathbb{R}}\eta\left(U|\bar{U}\right)dx  \\
		& +  C\delta_C(1+t)^{-1} \int_{\mathbb{R}}e^{-\frac{2C_1|x|^2}{1+t}}|(v-\bar{v},\theta-\bth)|^2 dx.
	\end{aligned}
\end{align}
	\noindent$\bullet$ Estimates on the terms $\bS_{i}(i=1,2)$: 	
It follows from \eqref{est:QC} that
\begin{align}\label{est_Q12}
	\|Q_1^C\|_{L^2(\mathbb{R})} \leq C\delta_C(1+t)^{-\frac{5}{4}}, \qquad \|Q_2^C\|_{L^2(\mathbb{R})} \leq C\delta_C(1+t)^{-\frac{7}{4}}.
\end{align}
Thus, it follows from Cauchy-Schwartz inequality,  Lemma \ref{lm:interaction} and \eqref{est_Q12} that
\begin{align}\label{est_S12}
	\begin{aligned}
		\bold{S}_1 + \bold{S}_2= &-\int_{\mathbb{R}} a Q_1(u-\bar{u}) dx  -\int_{\mathbb{R}} a\left(\frac{\theta}{\bar{\theta}}-1\right)Q_2 dx  \\
		\leq & C\int_{\mathbb{R}} (|Q^I_1| + |Q^C_1|)|u-\bar{u}| dx + C\int_{\mathbb{R}} (|Q^I_2| + |Q^C_2|)|\theta-\bth|dx \\
		\leq &  C(\|Q^I_1\|_{L^2(\mathbb{R})} + \|Q^C_1\|_{L^2(\mathbb{R})})\|u-\bar{u}\|_{L^2(\mathbb{R})} + C (\|Q^I_2\|_{L^2(\mathbb{R})} + \|Q^C_2\|_{L^2(\mathbb{R})})\|\theta-\bth\|_{L^2(\mathbb{R})} \\
		\leq & \ C\left[\delta_1(\delta_C + \delta_3)e^{-C\delta_1 t} + \delta_3(\delta_1 + \delta_C)e^{-C\delta_3 t}\right]\|(u-\bar{u}, \theta-\bth)\|_{L^2(\mathbb{R})} \\ 
        &+C\delta_C\left[(\delta_1 + \delta_3)e^{-C t}+(1+t)^{-\frac{5}{4}}\right]\|(u-\bar{u}, \theta-\bth)\|_{L^2(\mathbb{R})}. 
	\end{aligned}
\end{align}

	\noindent$\bullet$ (Summary): Using \eqref{est_total}, \eqref{est_Yij}, \eqref{est_Bi}, and \eqref{est_S12}, together with the smallness of $\delta_0, \delta_i, \delta_C, \varepsilon$, we get that
\begin{align*}
\begin{aligned}
		& {d \over dt} \int_{\mathbb{R}}a\bar{\theta}\eta\left(U|\bar{U}\right)dx    \\
		\leq &-\sum_{i\in\{1,3\}}\frac{\delta_i}{4M_i}|\dot{\bold{X}_i}|^{2}    - \frac{C}{2}\sum_{i\in\{1,3\}}\mathcal{G}_i^{S} -\frac{1}{8}\bold{D}(U) + C\delta_C(1+t)^{-1} \int_{\mathbb{R}}e^{-\frac{2C_1|x|^2}{1+t}}|(v-\bar{v},\theta-\bth)|^2 dx  \\
		& + \delta_1^{\frac{4}{3}}\left(\delta_3^{\frac{4}{3}} e^{-C\delta_3 t}  + \delta_C^{\frac{4}{3}}e^{-Ct}\right) +  \delta_3^{\frac{4}{3}}\left(\delta_1^{\frac{4}{3}} e^{-C\delta_1 t}  + \delta_C^{\frac{4}{3}}e^{-Ct}\right) \\ 
        & +  \left(\sum_{i\in\{1,3\}}C\delta^2_i\exp(-C\delta_it) + \frac{C}{t^2}\right)\int_{\mathbb{R}}a\bar{\theta}\eta\left(U|\bar{U}\right)dx   \\
		& +C\left[\delta_1(\delta_C + \delta_3)e^{-C\delta_1 t} + \delta_3(\delta_1 + \delta_C)e^{-C\delta_3 t}\right]\|(u-\bar{u}, \theta-\bth)\|_{L^2(\mathbb{R})} \\ 
        & +C\delta_C\left[(\delta_1 + \delta_3)e^{-C t}+(1+t)^{-\frac{5}{4}}\right]\|(u-\bar{u}, \theta-\bth)\|_{L^2(\mathbb{R})}.
	\end{aligned}
\end{align*}		
By applying Grönwall’s inequality, we obtain that for all $t \ge 1$,
\begin{align*}
	\begin{aligned}
		&  \int_{\mathbb{R}}a\bar{\theta}\eta\left(U|\bar{U}\right)dx + \int_{1}^{t}  \left(\sum_{i\in\{1,3\}}\frac{\delta_i}{4M_i}|\dot{\bold{X}_i}|^{2}    + \frac{C}{2}\sum_{i\in\{1,3\}}\mathcal{G}_i^{S} +\frac{1}{8}\bold{D}\right) \\
	    \leq  & \ \bigg[ \int_{\mathbb{R}}a\bar{\theta}\eta\left(U|\bar{U}\right)dx\big|_{t=1} +C\delta_C\int_{1}^{t}  (1+s)^{-1} \int_{\mathbb{R}}e^{-\frac{2C_1|x|^2}{1+s}}|(v-\bar{v},\theta-\bth)|^2 dxds   \\
	    &  \ \ \ + C(\delta_1 + \delta_C + \delta_3)\|(u-\bar{u}, \theta-\bth)\|_{L^{\infty}(0,T;L^2(\mathbb{R}))} \\ 
        & \ \ \ + C\delta_1^{\frac{4}{3}}\left(\delta_3^{\frac{4}{3}} e^{-C\delta_3 t}  + \delta_C^{\frac{4}{3}}e^{-Ct}\right) + C \delta_3^{\frac{4}{3}}\left(\delta_1^{\frac{4}{3}} e^{-C\delta_1 t}  + \delta_C^{\frac{4}{3}}e^{-Ct}\right)\bigg]     \\
		&  \quad \cdot \exp \left(\int_{1}^{t}\bigg(\sum_{i\in\{1,3\}}C\delta^2_i\exp(-C\delta_i s ) + \frac{C}{s^2}\bigg)ds \right)  \\
		 	\leq & \ \int_{\mathbb{R}}a\bar{\theta}\eta\left(U|\bar{U}\right)dx\big|_{t=1} +C\delta_C\int_{1}^{t}  (1+s)^{-1} \int_{\mathbb{R}}e^{-\frac{2C_1|x|^2}{1+s}}|(v-\bar{v},\theta-\bth)|^2 dxds   \\
		&  + C(\delta_1 + \delta_C + \delta_3)\|(u-\bar{u}, \theta-\bth)\|_{L^{\infty}(0,T;L^2(\mathbb{R}))}  \\
        & + C\delta_1^{\frac{4}{3}}\left(\delta_3^{\frac{4}{3}} e^{-C\delta_3 t}  + \delta_C^{\frac{4}{3}}e^{-Ct}\right) +  C\delta_3^{\frac{4}{3}}\left(\delta_1^{\frac{4}{3}} e^{-C\delta_1 t}  + \delta_C^{\frac{4}{3}}e^{-Ct}\right).     \\
	\end{aligned}
\end{align*}	

Subsequently, by observing that
\[\|U-\bar{U}\|^2_{L^2(\mathbb{R})} \sim  \int_{\mathbb{R}}a\bar{\theta}\eta\left(U|\bar{U}\right)dx, \qquad  \forall t \in [0,T], \]
we obtain
\begin{align}\label{est_bigtime}
	\begin{aligned}
		&  \int_{\mathbb{R}}a\bar{\theta}\eta\left(U|\bar{U}\right)dx + \int_{1}^{t}  \left(\sum_{i\in\{1,3\}}\frac{\delta_i}{4M_i}|\dot{\bold{X}_i}|^{2}    + \frac{C}{2}\sum_{i\in\{1,3\}}\mathcal{G}_i^{S} +\frac{1}{8}\bold{D}\right) \\
		\leq & \ \int_{\mathbb{R}}a\bar{\theta}\eta\left(U|\bar{U}\right)dx\big|_{t=1} +C\delta_C\int_{1}^{t}  (1+s)^{-1} \int_{\mathbb{R}}e^{-\frac{2C_1|x|^2}{1+s}}|(v-\bar{v},\theta-\bth)|^2 dxds   \\
		& \ + \delta_1^{\frac{4}{3}}\left(\delta_3^{\frac{4}{3}} e^{-C\delta_3 t}  + \delta_C^{\frac{4}{3}}e^{-Ct}\right) +  \delta_3^{\frac{4}{3}}\left(\delta_1^{\frac{4}{3}} e^{-C\delta_1 t}  + \delta_C^{\frac{4}{3}}e^{-Ct}\right).
	\end{aligned}
\end{align}

	\subsection{Estimate in small time}
From \eqref{est_RE} and \eqref{eq_XY}, we have
		\begin{align}\label{est_smalltime}
			{d \over dt}& \int_{\mathbb{R}}a\bar{\theta}\eta\left(U|\bar{U}\right)dx  \leq\mathcal{R}
		\end{align}

where \begin{align*}
	\mathcal{R} =  -\sum_{i\in\{1,3\}}\frac{\delta_i}{M_i}|\dot{\bold{X}}_i|^{2} + \sum_{i\in\{1,3\}}\left(\dot{\bold{X}}_i\sum_{j=4}^{6}\bold{Y}_{ij}\right) + \sum_{i=1}^{6}\bold{B}_i + \bold{S}_1 + \bold{S}_2 -\bold{G} -\bold{D},
\end{align*}
Applying Young’s inequality, we first obtain
\begin{align*}
		\mathcal{R}+ \sum_{i\in\{1,3\}}\frac{\delta_i}{4M_i}|\dot{\bold{X}}_i|^{2}  + \bold{D} + \mathcal{G}^S \leq   \sum_{i\in\{1,3\}}\left(\frac{C}{\delta_i}\sum_{j=4}^{6}|\bold{Y}_{ij}|^2\right) + \sum_{i=1}^{6}\bold{B}_i + \bold{S}_1 + \bold{S}_2 ,
\end{align*}
Using \eqref{apriori_small}, Lemma \ref{est:shock_prop}, and \eqref{est:ax} we have
\begin{align*}
	\begin{aligned}
		\sum_{j=4}^{6}|\bold{Y}_{ij}|\leq  &C\|(\widetilde{\theta}_i)_x\|_{L^{\infty}(\mathbb{R})} (\|v-\bar{v}\|^2_{L^2(\mathbb{R})} +(\|\theta-\bth\|^2_{L^2(\mathbb{R})} )  \\
		& + C\|(a_i)_x\|_{L^{\infty}(\mathbb{R})} (\|v-\bar{v}\|^2_{L^2(\mathbb{R})} + \|u-\bar{u}\|^2_{L^2(\mathbb{R})} +(\|\theta-\bth\|^2_{L^2(\mathbb{R})} ) \\
		\leq & C\varepsilon^2\delta^2_i,
	\end{aligned}
\end{align*}
which yields
\begin{align*}
	\sum_{i\in\{1,3\}}\left(\frac{C}{\delta_i}\sum_{j=4}^{6}|\bold{Y}_{ij}|^2\right)  \leq C\varepsilon^4\sum_{i\in\{1,3\}}\delta^3_i.
\end{align*}
Similarly, we have 
\begin{align*}
	\begin{aligned}
		\sum_{i=1}^{6}\bold{B}_i  \leq &  C\left(\sum_{i\in\{1,3\}}\|(a_i)_x\|_{L^{\infty}(\mathbb{R})}+\|(\bar{v})_x\|_{L^{\infty}(\mathbb{R})} +\|(\bar{u})_x\|_{L^{\infty}(\mathbb{R})} + \|(\bth)_x\|_{L^{\infty}(\mathbb{R})}\right) \\
		& \ \ \ \  \cdot\left(\|v-\bar{v}\|^2_{L^2(\mathbb{R})} + \|u-\bar{u}\|^2_{H^1(\mathbb{R})} +\|\theta-\bth\|^2_{H^1(\mathbb{R})}\right)  + \frac{1}{80} \bD \\
		\leq  & C\varepsilon^2\left(\sum_{i\in\{1,3\}}\delta^2_i + \delta_C\right)+ \frac{1}{80} \bD ,
	\end{aligned}
\end{align*}
and 
\begin{align*}
	\begin{aligned}
	\bold{S}_1 + \bold{S}_2
	\leq &  C(\|Q^I_1\|_{L^2(\mathbb{R})} + \|Q^C_1\|_{L^2(\mathbb{R})})\|u-\bar{u}\|_{L^2(\mathbb{R})} + C (\|Q^I_2\|_{L^2(\mathbb{R})} + \|Q^C_2\|_{L^2(\mathbb{R})})\|\theta-\bth\|_{L^2(\mathbb{R})}  \\
	\leq & C(\delta_1 +\delta_3 + \delta_C )\varepsilon,
\end{aligned}
\end{align*}
and 
\begin{align*}
	\mathcal{G}_i^S \leq C\sum_{i\in\{1,3\}}\|(\widetilde{v}_i)_x\|_{L^{\infty}(\mathbb{R})} (\|v-\bar{v}\|^2_{L^2(\mathbb{R})} + \|u-\bar{u}\|^2_{L^2(\mathbb{R})} +(\|\theta-\bth\|^2_{L^2(\mathbb{R})} ) 
	\leq C\varepsilon^2 \delta^2_i.
\end{align*}
Accordingly, the preceding estimates yield a preliminary bound, valid for any $\delta_1, \delta_C, \delta_3\in (0,\delta_0)$, 

\begin{align}\label{est_small}
	\mathcal{R}+ \sum_{i\in\{1,3\}}\frac{\delta_i}{4M}|\dot{\bold{X}_i}|^{2}  + \bold{D} + \sum_{i\in\{1,3\}}\mathcal{G}_i^S \leq  C\delta_0, \qquad t>0.
\end{align}

\subsection{Conclusion}

\noindent As a first step, we use \eqref{est_smalltime} together with \eqref{est_small} to obtain a coarse estimate for small time  $t\leq 1$：
	\begin{align*}
	{d \over dt}& \int_{\mathbb{R}}a\bar{\theta}\eta\left(U|\bar{U}\right)dx + \sum_{i\in\{1,3\}}\frac{\delta_i}{4M_i}|\dot{\bold{X}_i}|^{2}  + \bold{D} + \sum_{i\in\{1,3\}}\mathcal{G}_i^S\leq C\delta_0,
\end{align*}
which implies
	\begin{align}
    \begin{aligned} \label{est_smalltime1}
& \int_{\mathbb{R}}a\bar{\theta}\eta\left(U|\bar{U}\right)dx\bigg|_{t=1} + \int_{0}^{1}\left(\sum_{i\in\{1,3\}}\frac{\delta_i}{4M_i}|\dot{\bold{X}_i}|^{2}  + \bold{D} + \sum_{i\in\{1,3\}}\mathcal{G}_i^S\right) \\ 
\leq & \int_{\mathbb{R}}a\bar{\theta}\eta\left(U|\bar{U}\right)dx\bigg|_{t=0}+  C\delta_0.
\end{aligned}
\end{align}
Combining \eqref{est_bigtime} and \eqref{est_smalltime1}, we conclude that
\begin{align}\label{est:combine}
	\begin{aligned}
		&  \int_{\mathbb{R}}a\bar{\theta}\eta\left(U|\bar{U}\right)dx + \int_{0}^{t}  \left(\sum_{i\in\{1,3\}} \delta_i|\dot{\bold{X}_i}|^{2}    + \sum_{i\in\{1,3\}}\mathcal{G}_i^{S} +\bold{D}\right) \\
			\leq & \ \int_{\mathbb{R}}a\bar{\theta}\eta\left(U|\bar{U}\right)dx\big|_{t=0} +C\delta_0\int_{1}^{t}  (1+s)^{-1} \int_{\mathbb{R}}e^{-\frac{2C_1|x|^2}{1+s}}|(v-\bar{v},\theta-\bth)|^2 dxds   \\
		&+ C\delta_1^{\frac{4}{3}}\left(\delta_3^{\frac{4}{3}} e^{-C\delta_3 t}  + \delta_C^{\frac{4}{3}}e^{-Ct}\right) +C  \delta_3^{\frac{4}{3}}\left(\delta_1^{\frac{4}{3}} e^{-C\delta_1 t}  + \delta_C^{\frac{4}{3}}e^{-Ct}\right)+ C\delta_0.     \\
	\end{aligned}
\end{align}	
In the end, an application of the following Lemma \ref{lm_spacetimeCD} (with small $\delta_0$) yields
\begin{align*}
	\begin{aligned}
		&  \int_{\mathbb{R}}a\bar{\theta}\eta\left(U|\bar{U}\right)dx + \int_{0}^{t}  \left(\sum_{i\in\{1,3\}} \delta_i|\dot{\bold{X}_i}|^{2}    + \sum_{i\in\{1,3\}}\mathcal{G}_i^{S} +\bold{D}\right) ds \\
		\leq  & \ \int_{\mathbb{R}}a\bar{\theta}\eta\left(U|\bar{U}\right)dx\big|_{t=0} 
+ C\delta_1^{\frac{4}{3}}\left(\delta_3^{\frac{4}{3}} e^{-C\delta_3 t}  + \delta_C^{\frac{4}{3}}e^{-Ct}\right) +  C\delta_3^{\frac{4}{3}}\left(\delta_1^{\frac{4}{3}} e^{-C\delta_1 t}  + \delta_C^{\frac{4}{3}}e^{-Ct}\right)    \\
	& \ + C\delta_0 \int_{0}^{t}\left\|(v-\bar{v})_x\right\|^2ds+ C\delta_0, \qed     
	\end{aligned}
\end{align*}	
we have thus proved Lemma~\ref{lm_zeroorder}.
The lemma below is employed to bound the second term on the right-hand side of \eqref{est:combine}. Its proof is essentially the same as \cite[Appendix C]{KVW-NSF} and is therefore omitted.
\begin{lemma}\label{lm_spacetimeCD}
	It holds that for any $t\in (0, T]$,
	\begin{align}\label{est_spacetimeCD}
		\begin{aligned}
			&\int_{0}^{t}(1+s)^{-1}\int_{\mathbb{R}}e^{-\frac{2C_1|x|^2}{1+s}}|(v-\bar{v}, u-\bar{u}, \theta-\bth)|^2 dxds  \\
			\leq & \ C\sup\limits_{s\in\left[0, T\right]} \left\|U-\bar{U}\right\|_{L^2(\mathbb{R})}^2 + C\sum_{i\in\{1,3\}}\delta_i\int_{0}^{t}\left|\dot{\bold{X}}_i(s)\right|^2 ds+ C\int_{0}^{t}\left(\mathcal{G}_1^{S}  + \mathcal{G}_3^{S} + \bold{D}(U) \right)ds   \\
			& + C\int_{0}^{t}\left\|(v-\bar{v})_x\right\|_{L^2(\mathbb{R})}^2ds + C\delta_0.
		\end{aligned}
	\end{align}
\end{lemma}
\section{Higher Order Estimates} \label{higher order}

\begin{lemma} 
Assume the conditions of Proposition \ref{prop:main} hold. Then there exists a constant $C > 0$, independent of $\delta_0$, $\varepsilon$, and $T$, such that for every $t \in (0, T]$,
		\begin{align*}
		\begin{aligned}
			& \  \|(v-\bar{v}, u-\bar{u}, \theta- \bar{\theta})(t, \cdot)\|^2_{H^1(\mathbb{R})}  + \int_{0}^{t}  \left(\sum_{i\in\{1,3\}} \delta_i|\dot{\bold{X}_i}|^{2}    + \sum_{i\in\{1,3\}}\mathcal{G}_i^{S}\right)   \\
			&  +\int_{0}^{t}  \|(v-\bar{v}, u-\bar{u}, \theta-\bar{\theta})_x\|^2_{L^2(\mathbb{R})} ds +  \int_{0}^{t} \|(u-\bar{u}, \theta-\bar{\theta})_{xx}\|^2_{L^2(\mathbb{R})}  ds    \\
			\leq &   C\|(v-\bar{v}, u-\bar{u}, \theta- \bar{\theta})(0, \cdot)\|_{H^1(\mathbb{R})}+ C\delta_0^{\frac{1}{2}}.
		\end{aligned}
	\end{align*}
\end{lemma}

\begin{proof}
The detailed proof is omitted here, as it follows closely the one presented in \cite[Lemma 5.1]{KVW-NSF}.
 \end{proof}

\bibliographystyle{plain}
\bibliography{ref} 
\end{document}